\documentclass[10pt]{article}
\usepackage[a4paper]{geometry}
\usepackage{amsthm}
\theoremstyle{plain}
\newtheorem{theorem}{Theorem}[section]
\newtheorem{definition}[theorem]{Definition}
\newtheorem{lemma}[theorem]{Lemma}
\newtheorem{corollary}[theorem]{Corollary}
\theoremstyle{remark}
\newtheorem{remark}{Remark}[section]

\usepackage[breaklinks,bookmarks=false]{hyperref}
\hypersetup{colorlinks, linkcolor=blue, citecolor=blue,
urlcolor=blue, plainpages=false, pdfwindowui=false,
pdfstartview={FitH}}

\usepackage{amsmath,amsfonts,amssymb}
\usepackage{graphicx,verbatim,mathtools}
\usepackage[only,llparenthesis,rrparenthesis]{stmaryrd}
\usepackage{hyperref}
\usepackage{bm}

\newcommand{\vphi}{\varphi}
\newcommand{\dd}{\mathrm{d}}
\renewcommand{\dim}{d}
\newcommand{\abs}[1]{\lvert#1\rvert}
\newcommand{\tends}{\rightarrow}
\newcommand{\norm}[1]{\lVert#1\rVert}
\newcommand{\p}{\partial}

\renewcommand{\th}{{h\tau}}
\newcommand{\lld}{\llparenthesis}
\newcommand{\rrd}{\rrparenthesis}
\DeclareMathOperator{\diam}{diam}
\DeclareMathOperator{\Div}{div}

\DeclareMathOperator*{\argmin}{argmin}
\newcommand{\eval}[2]{\left. #1\right|_{#2}}
\newcommand{\pair}[2]{\langle #1,#2 \rangle}
\newcommand{\R}{\mathbb{R}}
\newcommand{\Om}{\Omega}
\newcommand{\om}{\omega}
\newcommand{\DO}{\partial\Om}
\newcommand{\calT}{\mathcal{T}}
\newcommand{\calE}{\mathcal{E}}
\newcommand{\calP}{\mathcal{P}}
\newcommand{\calI}{\mathcal{I}}
\newcommand{\calB}{\mathcal{B}}
\newcommand{\calR}{\mathcal{R}}
\newcommand{\calQ}{\mathcal{Q}}
\newcommand{\calV}{\mathcal{V}}

\newcommand{\LH}{L^2(0,T;H^1_0(\Om))}
\newcommand{\HHm}{H^1(0,T;H^{-1}(\Om))}
\newcommand{\By}{\calB_Y}
\newcommand{\RTN}{\mathbf{RTN}} 
\newcommand{\Hdiv}{\bm{H}(\Div,\Om)}
\newcommand{\Hdivoma}{\bm{H}(\Div,\oma)}
\newcommand{\CR}{\widetilde{\calT^{n}}}
\newcommand{\Ta}{\widetilde{\calT^{\ta,n}}}

\newcommand{\RTNa}{\mathbf{RTN}_{p_{\ta}}(\Ta)}
\newcommand{\Pa}{\calP_{p_{\ta}}(\Ta)}
\newcommand{\Pal}{\calP_{p_{\ta}-1}(\Ta)}
\newcommand{\RTNal}{\mathbf{RTN}_{p_{\ta}-1}(\Ta)}
\newcommand{\Vh}{V_h}
\newcommand{\Vt}{V_{h\tau}}
\newcommand{\Vn}{V^n_h}
\newcommand{\VCR}{\widetilde{\Vn}}
\newcommand{\Vnm}{V^{n-1}_h}
\newcommand{\elCR}{\widetilde{K}}
\newcommand{\Uth}{\calI u_{\th}}
\newcommand{\ta}{\mathbf{a}}
\newcommand{\psia}{\psi_{\ta}}
\newcommand{\Vthan}{\bm{V}_{h\tau}^{\ta,n}}
\newcommand{\Qthan}{Q_{h\tau}^{\ta,n}}
\newcommand{\bvthj}{\bm{v}_{h,j}^{\ta,n}}
\newcommand{\oma}{{\om_{\ta}}}

\newcommand{\Pia}{\Pi_{\th}^{\ta,n}}
\newcommand{\calVh}{\mathcal{V}^n}
\newcommand{\calVhint}{\mathcal{V}_{\mathrm{int}}^n}
\newcommand{\calVhext}{\mathcal{V}_{\mathrm{ext}}^n}
\newcommand{\Va}{\bm{V}^{\ta,n}_h}
\newcommand{\Qa}{Q_h^{\ta,n}}
\newcommand{\Ra}{\calR_{\th}^{\ta,n}}
\newcommand{\Raj}{\calR_{h,j}^{\ta,n}}
\newcommand{\stha}{\bm{\sigma}_{\th}^{\ta,n}}
\newcommand{\rtha}{r_{\th}^{\ta,n}}
\newcommand{\sth}{\bm{\sigma}_{\th}}
\newcommand{\sthj}{\bm{\sigma}_{h,j}^{\ta,n}}
\newcommand{\rthj}{r_{h,j}^{\ta,n}}
\newcommand{\tautha}{\bm{\tau}_{\th}^{\ta,n}}
\newcommand{\tauthj}{\bm{\tau}^{\ta,n}_{h,j}}
\newcommand{\gtautha}{g_{\th}^{\ta,n}}
\newcommand{\fthj}{f_{h,j}^{\ta,n}}
\newcommand{\bxithj}{\bxi_{h,j}^{\ta,n}}
\newcommand{\bxi}{\bm{\xi}}
\newcommand{\etaEq}{\eta_{\mathrm{F},K}^{n}}
\newcommand{\etaJ}{\eta_{\mathrm{J},K}^n}
\newcommand{\etaOscT}{\eta_{\mathrm{osc},\tau}}
\newcommand{\etaOscS}{\eta_{\mathrm{osc},h,K}^{n}}

\newcommand{\etaOsca}{\eta_{\mathrm{osc}}^{\ta,n}}

\newcommand{\etaCJ}{\eta_{\mathrm{C}}^{n}}

\newcommand{\etaOscTn}{\eta_{\mathrm{osc},\tau}^n}
\newcommand{\etaOscInit}{\eta_{\mathrm{osc},\mathrm{init}}}
\newcommand{\nEy}{\calE_Y}
\newcommand{\NEy}{\norm{\cdot}_{\nEy}}
\newcommand{\Ey}{\norm{u-u_{\th}}_{\nEy}}
\newcommand{\nEya}{\calE_Y^{\ta,n}}

\newcommand{\Eya}{\abs{u-u_{\th}}_{\nEya}}

\title{Guaranteed, locally space-time efficient, and polynomial-degree robust a posteriori error estimates for high-order discretizations of parabolic problems
\thanks{This project has received funding from the European Research Council (ERC) under the European Union's Horizon 2020 research and innovation program (grant agreement No 647134 GATIPOR).}}
\author{Alexandre~Ern\footnotemark[2] \and Iain~Smears\footnotemark[3] \and Martin~Vohral\'ik\footnotemark[3]}

\begin{document}

\numberwithin{equation}{section}

\renewcommand{\thefootnote}{\fnsymbol{footnote}}
\footnotetext[2]{Universit\'e Paris-Est, CERMICS (ENPC), 77455 Marne-la-Vall\'{e}e cedex 2, France (alexandre.ern@enpc.fr).}
\footnotetext[3]{INRIA Paris, 2 Rue Simone Iff, 75012 Paris, France (iain.smears@inria.fr, martin.vohralik@inria.fr)}
\renewcommand{\thefootnote}{\arabic{footnote}}

\maketitle

\begin{abstract}
We consider the a posteriori error analysis of approximations of parabolic problems based on arbitrarily high-order conforming Galerkin spatial discretizations and arbitrarily high-order discontinuous Galerkin temporal discretizations. Using equilibrated flux reconstructions, we present a posteriori error estimates for a norm composed of the $L^2(H^1)\cap H^1(H^{-1})$-norm of the error and the temporal jumps of the numerical solution. The estimators provide guaranteed upper bounds for this norm, without unknown constants. Furthermore, the efficiency of the estimators with respect to this norm is local in both space and time, with constants that are robust with respect to the mesh-size, time-step size, and the spatial and temporal polynomial degrees. We further show that this norm, which is key for local space-time efficiency, is globally equivalent to the $L^2(H^1)\cap H^1(H^{-1})$-norm of the error, with polynomial-degree robust constants.
The proposed estimators also have the practical advantage of allowing for very general refinement and coarsening between the timesteps.
\end{abstract}

\smallskip 


\section{Introduction}\label{sec:introduction}%
We consider the heat equation
\begin{equation}\label{eq:parabolic}
\begin{aligned}
\p_t u - \Delta u = f & & & \text{in }\Om\times(0,T),\\
 u = 0 & & &\text{on }\DO\times (0,T),\\
 u(0) = u_0 & & &\text{in }\Om,
\end{aligned}
\end{equation}
where $\Om \subset \R^d $, $1\leq d \leq 3$, is a bounded, connected, polyhedral open set with Lipschitz boundary, and $T>0$ is the final time. We assume that $f \in L^2(0,T;L^2(\Om))$, and that $u_0\in L^2(\Om)$. We are interested here in developing a posteriori error estimates for a class of high-order discretizations of~\eqref{eq:parabolic}. In particular, we consider a conforming finite element method (FEM) in space on unstructured shape-regular meshes, and a discontinuous Galerkin discretization in time, where one is free to vary the approximation orders $p$ in space and $q$ in time, as well as the mesh size $h$ and time-step size~$\tau$, leading to what we call a $hp$-$\tau q$ method. These methods are highly attractive from the point of view of flexibility, accuracy, and computational efficiency, since it is known from a priori analysis that judicious local adaptation of the discretization parameters can lead to exponential convergence rates with respect to the number of degrees of freedom, even for solutions with singularities near domain corners, edges, and at initial times~\cite{Schotzau2000,Schwab1998,WerderGerdesSchotzauSchwab2001}. In practice, it is desirable to determine the adaptation algorithmically, which requires rigorous and high-quality a posteriori error control in order to exploit the potential for high accuracy and efficiency of $hp$-$\tau q $ discretizations.
We recall that a posteriori error estimates should ideally give \emph{guaranteed} upper bounds on the error, i.e.\ without unknown constants, should be \emph{locally efficient}, meaning that the local estimators should be bounded from above by the error measured in a local neighbourhood, and, moreover, should be \emph{robust}, with all constants in the bounds being independent of the discretization parameters; we refer the reader to~\cite{Verfurth2013} for an introduction to these concepts.

In the context of parabolic problems, the a posteriori error analysis for low- and fixed-order methods has received significant attention over the past decade, with efforts mostly concentrated on fixed-order FEM in space coupled with an implicit Euler or Crank--Nicolson time-stepping scheme, leading to estimates for a wide range of norms. These include estimates for the $L^2(H^1)$-norm of the error considered independently by Picasso and Verf\"urth~\cite{Picasso1998,Verfurth1998}, with efficiency bounds typically requiring restrictions on the relation between the sizes of the time-steps and the meshes.
Estimates for the $L^2(H^1)\cap H^1(H^{-1})$-norm estimates were first considered by Verf\"urth in~\cite{Verfurth2003}, who crucially proved local-in-time yet global-in-space efficiency of estimators without restrictions between time-step and mesh sizes, see also Bergam, Bernardi, and Mghazli~\cite{BergamMghazli2005}. Guaranteed upper bounds were later obtained by Ern and Vohral\'ik in \cite{ErnVohralik2010}, with similar efficiency results as in \cite{Verfurth2003}.
There are also upper bounds in $L^2(L^2)$, $L^\infty(L^2)$ and $L^\infty(L^\infty)$ and higher order norms, based on either duality techniques as in Eriksson and Johnson~\cite{Eriksson1995} or the elliptic reconstruction technique originally due to Makridakis and Nochetto~\cite{Makridakis2003} and later considered in the fully discrete context by Lakkis and Makridakis~\cite{LakkisMakridakis2006}, see also \cite{LakkisMakridakisPryer2015} and the references therein. Repin~\cite{Repin2002} studied so-called functional estimates.
Finally, a posteriori error estimates developed in the context of the heat equation often serve as a starting point for extensions to diverse applications, including nonlinear problems and spatially non-conforming methods among others~\cite{DiPietroVohralikSoleiman2015,DolejsiErnVohralik2013,GeorgoulisLakkisVirtanen2011,Kreuzer2013,NicaiseSoualem2005}.
Adaptive algorithms for parabolic problems are studied in \cite{ChenFeng2004,GaspozKreuzerSiebertZiegler2016,KreuzerMollerSchmidtSiebert2012}.

It is apparent from the literature that, even for low- and fixed-order methods, there are remaining outstanding issues, particularly in terms of the efficiency of the estimators. The efficiency of the estimators is significantly influenced by the choice of norm to be estimated, with the strongest available results being attained by $Y$-norm estimates, where henceforth $Y\coloneqq L^2(H^1_0)\cap H^1(H^{-1})$. However, as mentioned above, even in this norm, the full space-time local efficiency of the estimators is not known. It is helpful to examine here more closely the issue of spatial locality of estimates in order to motivate the approach adopted in this work.
For example, let us momentarily consider an implicit Euler discretization in time and a conforming FEM in space, recalling that the implicit Euler method corresponds to the lowest-order discontinuous Galerkin time-stepping method, which uses piecewise constant approximations with respect to time. Since the resulting numerical solution $u_{\th}$ is discontinuous with respect to time, and thus $u_{\th}\notin Y$, it is usual to consider a reconstruction, denoted by $\Uth \in Y$, obtained by piecewise linear interpolation at the time-step nodes, and it is seemingly natural to seek a posteriori error estimates for $\norm{u-\Uth}_{Y}$, where $\norm{\cdot}_Y$ is defined in \eqref{eq:XYnorms} below, and where $u$ is the solution of~\eqref{eq:parabolic}; for example, this corresponds to the approach adopted in~\cite{Verfurth2003}.

The main issue encountered in studying the efficiency of estimators with respect to the error measured by $\norm{u-\Uth}_{Y}$ is that $\Uth$ fails to satisfy the Galerkin orthogonality property on the discrete level: instead, $\Uth$ satisfies
\begin{equation}\label{eq:discrete_residual}
\int_{I_n} (f,v_\th) - (\p_t\Uth,v_{\th}) - (\nabla \Uth,\nabla v_{\th}) \dd t  = \int_{I_n} (\nabla (u_{\th}-\Uth),\nabla v_{\th})\dd t,
\end{equation}
for all discrete test functions $v_{\th}$ that are constant in time over the given time interval $I_n$ and belong to the associated finite element space $\Vn$ (see section~\ref{sec:fem_approximation} for complete definitions). It is seen from the right-hand side of~\eqref{eq:discrete_residual} that a discrete residual arises from the difference between the numerical solution $u_{\th}$ and its reconstruction $\Uth$.
Since it does not appear possible to show that the discrete residual in~\eqref{eq:discrete_residual} is controlled locally in space and time by the corresponding local space-time components of the $Y$-norm of the error, one cannot obtain local efficiency of the estimators with respect to this norm; we note that this issue remains essentially independent of the specific construction of the a posteriori error estimators, whether they are residual-type estimators as in~\cite{Verfurth2003} or equilibrated flux estimators as we consider here. Nevertheless, Verf\"urth~\cite{Verfurth2003} showed in the lowest-order case that the global-in-space local-in-time norm of the discrete residual can be bounded by the corresponding global-in-space local-in-time $Y$-norm of $u-\Uth$, which leads to time-local yet space-global efficiency of the estimators.
In order to overcome the issue of the loss of spatial locality, we observe that the discrete residual in \eqref{eq:discrete_residual} is equivalent to the temporal jumps in the numerical solution $u_{\th}$, which is an error component in itself since it measures the lack of conformity in $Y$ of the numerical solution $u_{\th}$.
It is therefore natural to consider a composite norm $\Ey$ that includes both $\norm{u-\Uth}_Y$ and the norm of jumps of the numerical solution $u_{\th}$. As we explain below, we then recover the fully space-time local efficiency of our estimators with respect to $\Ey$: see \eqref{eq:intro_lower} below, and see Theorem~\ref{thm:Y_norm_guaranteed_efficiency} of section~\ref{sec:Y_aposteriori}.

For $hp$-FEM discretizations, one of the key issues concerns the robustness of the estimators with respect to the polynomial degree; this issue appears already in the context of elliptic problems, where Melenk and Wohlmuth \cite{MelenkWolhmuth2001} and Melenk~\cite{Melenk2005} showed that the well-known residual estimators fail to be polynomial-degree robust.
In a breakthrough work, Braess, Pillwein,~and~Sch\"oberl~\cite{Braess2009} established the polynomial-degree robustness of estimators based on equilibrated fluxes, in the context of elliptic diffusion problems.
These estimators are based on a globally $\bm{H}(\Div)$-conforming flux constructed from mixed finite element approximations of local Neumann problems over vertex-centred patches of the mesh.
The polynomial degree robustness of these estimators was then recently generalized to nonconforming and mixed methods for elliptic problems in \cite{ErnVohralik2015}, to which we refer the reader for further references on the literature of equilibrated flux estimators for elliptic problems. In the context of parabolic problems, we must also address the additional question of robustness of the estimators with respect to the temporal polynomial degrees.
In comparison to low- and fixed-order methods, there are comparatively few works on a posteriori error estimates for high-order discretizations of parabolic problems.  Building on the earlier work of Makridakis and Nochetto~\cite{Makridakis2006}, Sch\"otzau and Wihler~\cite{Schotzau2010} studied the effect of the temporal approximation order of a posteriori estimates for a composite norm of $L^\infty(L^2)\cap L^2(H^1)$-type, in the context of high-order temporal semi-discretizations of abstract evolution equations by the discontinuous and continuous Galerkin time-stepping methods. Otherwise, it appears that a posteriori error estimates for $hp$-$\tau q$ discretizations of parabolic problems remain essentially untouched.

In this work, we present guaranteed, locally space-time efficient, and polynomial degree robust a posteriori error estimators for $hp$-$\tau q$ discretizations of parabolic problems. This is by no means simple, as it requires the treatment of the challenges that have been outlined above. Our main results are the following.

Let the spaces $Y \coloneqq L^2(H^1_0)\cap H^1(H^{-1})$ and $X\coloneqq L^2(H^1_0)$ be respectively equipped with their standard norms $\norm{\cdot}_Y$ and $\norm{\cdot}_X$ defined in \eqref{eq:XYnorms} below.
Let $Y+\Vt$ be the sum of the continuous and approximate solution spaces, recalling that $\Vt\subset X$ and that $u_{\th}\notin Y$ due to the temporally discontinuous approximation.
Let $\calI \colon Y+\Vt \tends Y$ be the reconstruction operator defined in section~\ref{sec:reconstruction} below, where we note that $\calI v = v$ if and only if $v\in Y$. Let the norm $\norm{\cdot}_{\nEy}$ be defined by $\norm{v}_{\nEy}^2 \coloneqq \norm{\calI v}_{Y}^2+\norm{v-\calI v}_X^2$ for all $v \in Y+\Vt$.

\paragraph{Guaranteed upper bounds}
In Theorem~\ref{thm:Y_norm_guaranteed_efficiency} of section~\ref{sec:Y_aposteriori}, we show a posteriori estimates in the norm $\norm{\cdot}_{\nEy}$. In particular, we have $\calI u = u$, so $\Ey^2=\norm{u-\Uth}_Y^2+\norm{u_{\th}-\Uth}_{X}^2$, where we note that $\norm{u_{\th}-\Uth}_{X}$ is a measure of the temporal jumps of the numerical solution~$u_{\th}$. In the absence of data oscillation, our bound takes the simple form
\begin{equation}\label{eq:intro_upper}
\Ey^2   \leq \sum_{n=1}^N \sum_{K\in\calT^n} \left\{ \int_{I_n} \norm{\sth + \nabla \Uth}_K^2 + \norm{\nabla(u_{\th}-\Uth)}_K^2 \,\dd t \right\},
\end{equation}
where $\sth$ is the $\bm{H}(\Div)$-conforming reconstruction; see sections~\ref{sec:fem_approximation} and~\ref{sec:flux_equilibration} for full definitions of the notation and construction of the estimators.

\paragraph{Polynomial-degree robustness and local space-time efficiency}
We establish local space-time efficiency of our estimators with polynomial-degree robust constants, expressed by the lower bound
\begin{equation}\label{eq:intro_lower}
 \int_{I_n} \norm{\sth + \nabla \Uth}_K^2 + \norm{\nabla(u_{\th}-\Uth)}_K^2 \dd t \lesssim \sum_{\ta\in\calV_{K}} \Eya^2  + \text{oscillation},
\end{equation}
where $K$ is an element of the mesh $\calT^n$ for time-step $I_n$, where $\calV_{K}$ denotes the set of vertices of $K$, and where $\Eya$ is the local component of $\Ey$ on the patch associated with the vertex $\ta$. Here, and in the following, the notation $a\lesssim b$ means that $a\leq C b$, with a constant $C$ that depends possibly on the shape-regularity of the spatial meshes, but is otherwise independent of the mesh-size, time-step size, as well as the spatial and temporal polynomial degrees. We stress that this efficiency bound does not require any relation between the sizes of the time-step and the mesh. The full bound is stated in Theorem~\ref{thm:Y_norm_guaranteed_efficiency} below.

In addition to the above results, the estimators proposed here are advantageous in terms of flexibility, since they do not require restrictions on coarsening or refinement between time-steps that appeared in earlier works, such as the transition condition used in \cite[p.~196, 201]{Verfurth2003}. The main tool to avoid this condition is Lemma~\ref{lem:p_robust_lifting} below.

\paragraph{Relation between $\Ey$ and $\norm{u-\Uth}_Y$}
In Theorem~\ref{thm:Y_norm_relation}, we simplify and generalize to higher-order temporal discretizations a key result of Verf\"urth~\cite{Verfurth2003}, namely that the jumps in the numerical solution can be controlled locally-in-time and globally-in-space by the $Y$-norm of $u-\Uth$. Specifically, for arbitrary approximation orders and for any time-step interval $I_n$, we show that
\[
\int_{I_n} \norm{\nabla(u_{\th}-\Uth)}^2\,\dd t \lesssim \int_{I_n} \norm{\p_t(u-\Uth)}_{H^{-1}(\Om)}^2 + \norm{\nabla (u-\Uth)}^2\,\dd t + \text{oscillation},
\]
where the constant, which is in fact known explicitly, is independent of all other quantities, including the temporal polynomial degree. The associated oscillation term involves the minimum of the source term data oscillation and the coarsening error, and thus can be controlled in practice. In the absence of this oscillation, we have
\begin{equation}\label{eq:intro_error_measure_equivalence}
\norm{u-\Uth}_{Y}\leq \Ey \leq 3 \norm{u-\Uth}_{Y},
\end{equation}
in addition to the upper and lower bounds~\eqref{eq:intro_upper}~and~\eqref{eq:intro_lower}. The key implication is that $\Ey$ and $\norm{u-\Uth}_{Y}$ are globally equivalent, although their local distributions may differ. We also stress that the equivalence is independent of the polynomial degrees. We offer some more refined equivalence results in section~\ref{sec:Y_norm_relation} in order to treat the case of possibly non-vanishing oscillation terms.

This paper is organized as follows. First, in section~\ref{sec:infsup} we introduce a functional setting for the a posteriori error analysis. We find it worthwhile to provide a complete derivation of the inf-sup analysis of the problem, as we give here quantitatively sharp results that are advantageous for the efficiency of the estimators in practice. Section~\ref{sec:fem_approximation} defines the setting in terms of notation, finite element approximation spaces, and the numerical scheme. Then, in section~\ref{sec:flux_equilibration}, we define the equilibrated flux reconstruction used in the a posteriori error estimates.
In section~\ref{sec:Y_aposteriori} we gather our main results underlying \eqref{eq:intro_upper}, \eqref{eq:intro_lower} and \eqref{eq:intro_error_measure_equivalence}. The proofs of the main results are treated in the subsequent sections: section~\ref{sec:Y_norm_relation} establishes the relation between $\Ey$ and $\norm{u-\Uth}_Y$; the proof of the guaranteed upper bound is given in section~\ref{sec:proof_upper}; and the efficiency of the estimators is the subject of section~\ref{sec:efficiency_Y_norm}.

\section{Inf-sup theory}\label{sec:infsup}
Recall that $\Om \subset \R^d $, $1\leq d \leq 3$ is a bounded, connected, polyhedral open set with Lipschitz boundary. For an arbitrary open subset $\omega\subset \Om$, we use $(\cdot,\cdot)_{\om}$ to denote the $L^2$-inner product for scalar- or vector-valued functions on $\omega$, with associated norm $\norm{\cdot}_{\om}$. In the special case where $\om = \Om$, we drop the subscript notation, i.e.\ $\norm{\cdot}\coloneqq\norm{\cdot}_{\Om}$.
We consider the function spaces $X \coloneqq \LH$ and $Y \coloneqq \LH \cap \HHm$,
which we equip with the following norms:
\begin{equation}\label{eq:XYnorms}
\begin{aligned}
\norm{\vphi}_{Y}^2 & \coloneqq \int_0^T \norm{\p_t \vphi}_{H^{-1}(\Om)}^2 + \norm{\nabla \vphi}^2 \,\dd t + \norm{\vphi(T)}^2 & & \forall\, \vphi \in Y,\\
\norm{v}_X^2 & \coloneqq \int_0^T \norm{\nabla v}^2 \,\dd t & & \forall\, v \in X.
\end{aligned}
\end{equation}
Define the bilinear form $\By \colon Y\times X\tends \R$ by
\begin{equation}
\By(\vphi,v) \coloneqq \int_0^T \pair{\p_t \vphi}{v} + (\nabla \vphi,\nabla v) \,\dd t,
\end{equation}
where $\vphi\in Y$ and $v\in X$ are arbitrary functions, and $\pair{\cdot}{\cdot}$ denotes here the duality pairing between $H^{-1}(\Om)$ and $H^1_0(\Om)$.
Then, the problem~\eqref{eq:parabolic} admits the following weak formulation: find $u \in Y$ such that $u(0)=u_0$ and such that
\begin{equation}\label{eq:Y_formulation}
\begin{aligned}
\By(u,v) = \int_0^T (f,v)\,\dd t & && \forall\, v \in X.
\end{aligned}
\end{equation}
The well-posedness of~\eqref{eq:Y_formulation} is well-known and can be shown by Galerkin's method \cite{Evans1998,Wloka1987}.
The following result states an inf--sup stability result for the bilinear form~$\By$ for the above spaces equipped with their respective norms. The inf--sup stability result presented here has the interesting and important property of taking the form of an identity, which is advantageous for the sharpness of a~posteriori error analysis.
\begin{theorem}[Inf--sup identity]\label{thm:inf_sup_parabolic}
For every $\vphi \in Y$, we have
\begin{equation}
\begin{aligned}
\norm{\vphi}_{Y}^2 &= \left[\sup_{v \in X\setminus\{0\}} \frac{ \By(\vphi,v) }{\norm{v}_{X}}\right]^2 + \norm{\vphi(0)}^2 . \label{eq:infsup_continuous_Y}
\end{aligned}
\end{equation}
\end{theorem}
\begin{proof}
For a fixed $\vphi \in Y$, let $w_*\in X$ be defined by $(\nabla w_*, \nabla v ) = \pair{\p_t \vphi}{v}$ for all $v\in H^1_0(\Om)$, a.e.\ in $(0,T)$, which implies the identity $\norm{\nabla w_*}^2 = \norm{\p_t \vphi}_{H^{-1}(\Om)}^2$ a.e. in~$(0,T)$. Furthermore, we have $\calB_Y(\vphi,v)=\int_0^T(\nabla(w_*+\vphi),\nabla v)\,\dd t$, thus implying that $\sup_{v\in X\setminus\{0\}} \By(\vphi,v)/\norm{v}_X = \norm{w_*+\vphi}_X$.
Thus, we obtain the desired identity~\eqref{eq:infsup_continuous_Y} by expanding the square
\begin{equation}\label{eq:lem_Y_norm_equivalence_2}
\begin{split}
\left[\sup_{v \in X\setminus\{0\}} \frac{ \By(\vphi,v) }{\norm{v}_{X}}\right]^2
&= \int_0^T \norm{\nabla( w_* + \vphi )}^2 \,\dd t\\
& = \int_0^T \norm{\nabla w_*}^2 + 2(\nabla w_*,\nabla \vphi) + \norm{\nabla \vphi}^2 \,\dd t \\
& = \int_0^T \norm{\p_t \vphi}_{H^{-1}(\Om)}^2 + 2\pair{\p_t \vphi}{\vphi} + \norm{\nabla \vphi}^2 \,\dd t \\
&= \norm{\vphi}_Y^2 - \norm{\vphi(0)}^2,
\end{split}
\end{equation}
where we note that we have used the identity $\int_{0}^T 2 \pair{\p_t \vphi}{\vphi}\,\dd t = \norm{\vphi(T)}^2 - \norm{\vphi(0)}^2$.
\end{proof}

In order to estimate the error between the solution $u$ of~\eqref{eq:parabolic} and its approximation, we define the residual functional~$\calR_Y \colon Y\tends X^\prime$ by
\begin{equation}\label{eq:R_Y_def}
\pair{\calR_Y(\vphi)}{v} \coloneqq \calB_Y(u-\vphi,v)= \int_0^T (f,v) - \pair{\p_t \vphi}{v} - (\nabla \vphi,\nabla v) \,\dd t,
\end{equation}
where $v\in X$ and $\vphi \in Y$, and $\pair{\cdot}{\cdot}$ denotes here the duality pairing between the dual space $X^\prime$ and $X$.
The dual norm of the residuals is naturally defined by $\norm{\calR_Y(\vphi)}_{X^\prime}\coloneqq \sup_{v\in X\setminus\{0\}} \tfrac{\pair{\calR_Y(\vphi)}{v}}{\norm{v}_X}$.
Theorem~\ref{thm:inf_sup_parabolic} implies the following equivalence between the error and dual norm of the residual: for all $\vphi \in Y$, we have
\begin{equation}
\norm{u-\vphi}_{Y}^2 = \norm{\calR_Y(\vphi)}_{X^\prime}^2 + \norm{u_0-\vphi(0) }^2. \label{eq:Y_error_residual_equivalence}
\end{equation}

\section{Finite element approximation}\label{sec:fem_approximation}

Consider a partition of the interval~$(0,T)$ into time-step intervals $I_n\coloneqq (t_{n-1},t_n) $, with $1\leq n \leq N$, where it is assumed that $[0,T]=\bigcup_{n=1}^N\overline{I_n}$, and that $\{t_n\}_{n=0}^N$ is strictly increasing with $t_0=0$ and  $t_N = T$.
For each interval $I_n $, we let $\tau_n \coloneqq t_n-t_{n-1}$ denote the local time-step size. We will not need any special assumptions about the relative sizes of the time-steps to each other. We associate a temporal polynomial degree $q_n\geq 0$ to each time-step $I_n$, and we gather all the polynomial degrees in the vector $\bm q = (q_n)_{n=1}^N$.
For a general vector space $V$, we shall  write $\calQ_{q_n}\left(I_n;V\right)$ to denote the space of $V$-valued univariate polynomials of degree at most $q_n$ over the time-step interval $I_n$.

\subsection{Meshes}
We associate a matching simplicial mesh $\calT^n$ of the domain $\Om$ for each $0\leq n \leq N$, where we assume shape-regularity of the meshes uniformly over all time-steps. This allows us to treat many applications where the meshes are obtained by refinement or coarsening between time-steps. We consider here only matching simplicial meshes for simplicity, although we indicate that mixed simplicial--parallelepipedal meshes, possibly containing hanging nodes, can be also be treated: see \cite{DolejsiErnVohralik2016} for instance.
The mesh $\calT^0$ will be used to approximate the initial datum $u_0$.
For each element~$K\in\calT^n$, let $h_K\coloneqq \diam K$ denote the diameter of $K$. We associate a local spatial polynomial degree $p_K\geq 1$ to each $K\in \calT^n$, and we gather all spatial polynomial degrees in the vector $\bm p_n= (p_K)_{K\in\calT^n}$. In order to keep our notation sufficiently simple, the dependence of the local spatial polynomial degrees $p_K$ on the time-step is kept implicit, although we bear in mind that the polynomial degrees may change between time-steps.

\subsection{Approximation spaces}
For a general matching simplicial mesh $\calT$ with associated vector of polynomial degrees $\bm p=(p_K)_{K\in\calT}$, $p_K\geq 1$ for all $K\in\calT$, the $H^1_0(\Om)$-conforming $hp$-finite element space $\Vh(\calT,\bm p)$ is defined by
\begin{equation}\label{eq:conforming_space_def}
\Vh (\calT,\bm p)\coloneqq \left\{v_h \in H^1_0(\Om),\; \eval{v_h}{K} \in \calP_{p_K}(K)\quad\forall\,K\in\calT\right\},
\end{equation}
where $\calP_{p_K}(K)$ denotes the space of polynomials of total degree at most $p_K$ on $K$.
For shorthand, we denote $\Vn \coloneqq V_h(\calT^n,\bm{p}_n)$ for each $0\leq n \leq N$. Let $\Pi_{h} u_0 \in V_h^0$ denote an approximation to the initial datum $u_0$, a typical choice being the $L^2$-orthogonal projection onto $V^0_h$.
Given the collection of timesteps $\{I_n\}_{n=1}^N$, the vector $\bm q$ of temporal polynomial degrees, and the $hp$-finite element spaces $\{\Vn\}_{n=1}^N$, the spatio-temporal finite element space $\Vt$ is defined by
\begin{equation}\label{eq:space_time_fem}
\Vt \coloneqq \left\{ v_{\th}|_{(0,T)}\in X,\; \eval{v_{\th}}{I_n} \in \calQ_{q_n}(I_n;\Vn) \quad\forall\, n=1,\dots,N,\; v_{\th}(0)\in V_h^0 \right\}.
\end{equation}
Functions in $\Vt$ are generally discontinuous with respect to the time-variable at the partition points, although we take them to be left-continuous: for all $1\leq n \leq N$, we define $v_{\th}(t_n)$ as the trace at $t_n$ of the restriction $\eval{v_{\th}}{I_n}$. Functions in $\Vt$ are thus left-continuous; moreover they also have a well-defined value at $t_0=0$.
For all $0\leq n < N$, we denote the right-limit of $v_{\th}\in\Vt$ at $t_n$ by $v_{\th}(t_n^+)$.
Then, the temporal jump operators $\lld \cdot \rrd_n$, $0\leq n \leq N-1$, are defined on $\Vt$ by
\begin{equation}\label{eq:jump_operators}
\lld v_{\th} \rrd_n \coloneqq v_{\th}(t_n)-v_{\th}(t_n^+), \quad 0\leq n \leq N-1.
\end{equation}

\subsection{Refinement and coarsening}
Similary to other works, e.g., \cite[p.~196]{Verfurth2003}, we assume that we have at our disposal a common refinement mesh $\CR$ of $\calT^{n-1}$ and $\calT^n$ for each $1\leq n \leq N$, as well as associated polynomial degrees $\widetilde{\bm{p}}_n=(p_{\elCR})_{\elCR\in\CR}$, such that $\Vnm + \Vn \subset \VCR \coloneqq  \Vh(\CR,\widetilde{\bm{p}}_n) $.
For a function $v_{\th}\in \Vt$, we observe that $\lld v_{\th} \rrd_{n-1} \in \VCR$ for each $1\leq n \leq N$ since $v_{\th}(t_{n-1})\in \Vnm$, $v_{\th}(t_{n-1}^+)\in \Vn$, and $\Vnm + \Vn \subset \VCR$.
It is assumed that $\CR$ has the same shape-regularity as $\calT^{n-1}$ and $\calT^n$, and that every element $\elCR\in\CR$ is wholly contained in a single element $K^\prime\in\calT^{n-1}$ and a single element $K^{\prime\prime}\in\calT^n$. We emphasize that we do not require any assumptions on the relative coarsening or refinement between successive spaces $\Vnm$ and $\Vn$. We note that in the present context, refinement and coarsening can be obtained by modification of the meshes as well as change in the polynomial degrees. Concerning the polynomial degrees, we may choose for example $p_{\elCR} = \max(p_{K^\prime},p_{K^{\prime\prime}})$.
In the case where $\Vn$ is obtained from $\Vnm$ by refinement without coarsening, then we may choose $\CR \coloneqq \calT^n$ and $\widetilde{\bm{p}}_n \coloneqq \bm{p}_n$ so that $\VCR=\Vn$. However, we do not need the transition condition assumption from~\cite[p.~196, 201]{Verfurth2003}, which requires a uniform bound on the ratio of element sizes between $\CR$ and $\calT^n$.

\subsection{Numerical scheme}
The numerical scheme for approximating the solution of the parabolic problem~\eqref{eq:parabolic} consists of finding  $u_{\th} \in \Vt $ such that $u_{\th}(0)=\Pi_h u_0$, and, for each time-step interval $I_n$,
\begin{multline}\label{eq:num_scheme}
\int_{I_n} (\p_t u_{\th}, v_{\th} ) + (\nabla u_{\th}, \nabla v_{\th}) \,\dd t - \left(\lld u_{\th} \rrd_{n-1}, v_{\th}(t_{n-1}^+) \right) \\
= \int_{I_n} (f,v_{\th}) \,\dd t \qquad  \forall\,v_\th \in \calQ_{q_n}(I_n;\Vn).
\end{multline}
Here the time derivative $\p_t u_{\th}$ is understood as the piecewise time-derivative on each time-step interval $I_n$.
The numerical solution $u_{\th}\in \Vt$ can thus be obtained by solving the fully discrete problem~\eqref{eq:num_scheme} on each successive time-step. At each time-step, this requires solving a linear system that is symmetric only in the lowest-order case; this can be performed efficiently in practice for arbitrary orders, see~\cite{Smears2015a} and the references therein.

\subsection{Reconstruction operator}\label{sec:reconstruction} For each time-step interval $I_n$ and each nonnegative integer $q$, let $L_q^n$ denote the polynomial on $I_n$ obtained by mapping the standard $q$-th Legendre polynomial under an affine transformation of $(-1,1)$ to $I_n$. It follows that $L_q^n(t_n) =1$ for all $q\geq 0$, and $L_q^n(t_{n-1})=(-1)^q$, and that the mapped Legendre polynomials $\{L_q^n\}_{q\geq 0}$ are $L^2$-orthogonal on $I_n$, and satisfy $\int_{I_n}\abs{L_q^n}^2\,\dd t = \frac{\tau_n}{2q+1}$ for all $q\geq 0$.
We introduce the Radau reconstruction operator $\calI$ defined on $\Vt$ by
\begin{equation}\label{eq:def_radau_reconstruction}
\begin{aligned}
\eval{(\calI v_{\th})}{I_n} \coloneqq \eval{ v_{\th} }{I_n} +  \frac{(-1)^{q_n} }{2} \left( L_{q_n}^n - L_{q_n+1}^n \right) \lld v_{\th} \rrd_{n-1} & & & \forall\,v_{\th}\in\Vt.
\end{aligned}
\end{equation}
It is clear that $\calI$ is a linear operator on $\Vt$.
It follows from the properties of the Legendre polynomials that $\eval{\calI  v_{\th} }{I_n}(t_n) = v_{\th}(t_n)$, and that $ \eval{\calI v_{\th} }{I_n}(t_{n-1}^+) = v_{\th}(t_{n-1})$ for all $1\leq n\leq N$. Therefore, $\calI v_{\th}$ is continuous with respect to the temporal variable at the interval partition points $\{t_n\}_{n=0}^{N-1}$, and thus we have
\begin{equation}
\begin{aligned}
\calI v_{\th} \in H^1(0,T;H^1_0(\Om))\subset Y, \quad\eval{ \calI v_{\th} }{I_n} \in   \calQ_{q_n+1}\big( I_n;\VCR\big) & & & \forall \,v_{\th}\in\Vt,
\end{aligned}
\end{equation}
where we recall that $\Vnm + \Vn \subset \VCR$.
We easily deduce the following property of the reconstruction operator $\calI$ from integration-by-parts and the orthogonality of the polynomials $L^n_{q_n}$ and $L^n_{q_n+1}$ to all polynomials of degree strictly less than $q_n$ on the time-step interval $I_n$:
\begin{equation}\label{eq:radau_identity}
\begin{aligned}
\int_{I_n} \p_t \calI  v_{\th}\, \phi \,\dd t = \int_{I_n} \p_t v_{\th}\, \phi \,\dd t - \lld v_{\th} \rrd_{n-1} \phi(t_{n-1}^+) \quad \forall\,\phi \in \calQ_{q_n}(I_n;\R),
\end{aligned}
\end{equation}
where equality holds in the above equation in the sense of functions in $\VCR$.
We may therefore use \eqref{eq:radau_identity} to rewrite the numerical scheme~\eqref{eq:num_scheme} as
\begin{equation}\label{eq:num_scheme_equiv}
\int_{I_n} (\p_t  \Uth,v_{\th}) + (\nabla u_{\th}, \nabla v_{\th})
\,\dd t = \int_{I_n} (f,v_{\th}) \,\dd t \quad  \forall\,v_\th \in \calQ_{q_n}(I_n;\Vn).
\end{equation}
Note also that $\Uth(0)=\Pi_h u_0$.

\begin{remark}[Alternative equivalent definitions]\label{rem:radau_interpolant}
The operator $\calI$ is the Radau reconstruction operator commonly used in the a~posteriori error analysis of the DG time-stepping method {\upshape\cite{Makridakis2006}} and in the a~priori error analysis of time-dependent first-order PDEs {\upshape\cite{ErnSchieweck2016}}. Several equivalent definitions of $\calI$ have appeared in the literature, although it will be particularly advantageous for our purposes to use the definition~\eqref{eq:def_radau_reconstruction} of $\calI$ due to {\upshape\cite{Smears2015a}}, to which we refer the reader for further discussion on the equivalence of the various definitions.
\end{remark}

\begin{remark}[Extensions of $\calI$ to $Y+\Vt$]\label{rem:radau_extension}
In what follows, it will be helpful to extend $\calI$ to a linear operator over $Y+\Vt$. Note that the definition of the jump operators \eqref{eq:jump_operators} can be naturally extended to $Y+\Vt$, and therefore the definition \eqref{eq:def_radau_reconstruction} also extends naturally to $Y+\Vt$. In particular, $\calI\colon Y+\Vt\tends Y$, and we have $\calI\vphi = \vphi$ if and only if $\vphi\in Y$, since the jumps of any~$\vphi\in Y$ vanish identically.
\end{remark}


\section{Construction of the equilibrated flux}\label{sec:flux_equilibration}
The a posteriori error estimates presented in this paper are based on a discrete and locally computable $\bm{H}(\Div)$-conforming flux~$\sth$ that satisfies the key equilibration property
\begin{equation}\label{eq:sigma_th_equilibration}
\begin{aligned}
\p_t \Uth + \nabla{\cdot} \sth  = f_{\th} & &  &\text{in }\Om\times(0,T),
\end{aligned}
\end{equation}
where $\Uth$ is defined in section~\ref{sec:reconstruction}, and $f_{\th}\approx f$ is a data approximation defined in~\eqref{eq:f_discrete_approx} below. We call $\sth$ an equilibrated flux. We consider here the natural extension of existing flux reconstructions for elliptic problems \cite{Braess2009,BraessSchoberl2008,DestuynderMetivet1999,ErnVohralik2015} to the parabolic setting; see also \cite{DolejsiRoskovecVlasak2016}. In particular, for each time-step, $\sth$ is obtained as a sum of fluxes computed by solving local mixed finite element problems over the vertex-based patches of the current mesh, see Definition~\ref{def:flux_construction_1} of section~\ref{sec:flux_reconstruction_def} below.

\subsection{Local mixed finite element spaces}
We now define the mixed finite element spaces that are required for the construction of the equilibrated flux.
For each $1\leq n \leq N$, let $\calVh$ denote the set of vertices of the mesh $\calT^n$, where we distinguish the set of interior vertices $\calVhint$ and the set of  boundary vertices $\calVhext$. For each $\ta \in \calVh$, let $\psia$ denote the hat function associated with $\ta$, and let $\oma$ denote the interior of the support of $\psia$, with associated diameter $h_{\oma}$.
Furthermore, let $\Ta$ denote the restriction of the mesh $\CR$ to~$\oma$.
Recalling that the common refinement spaces $\VCR$ were obtained with a vector of polynomial degrees $\widetilde{\bm{p}}_n = (p_{\elCR})_{\elCR\in \CR}$, we associate to each $\ta \in \calVh$ the fixed polynomial degree
\begin{equation}\label{eq:patch_polynomial_degree}
p_{\ta} \coloneqq \max_{\elCR\in \Ta} (p_{\elCR}+1).
\end{equation}
Observe that $\psia \p_t\Uth|_{\elCR\times I_n}$ is a polynomial function with degree at most $q_n$ in time and at most $ p_{\ta}$ in space for each $\elCR\in \Ta$, $1\leq n \leq N$.

For a polynomial degree $p\geq 0$, let the local spaces $\calP_{p}(\Ta)$ and $\RTN_p(\Ta)$ be defined by
\begin{align*}
\calP_{p}(\Ta) &\coloneqq \{ q_h \in L^2(\oma),\quad q_h|_{\elCR} \in \calP_{p}(\Ta)\quad\forall\,\elCR\in\Ta\}
\\  \RTN_p(\Ta) &\coloneqq \{ \bm{v}_h \in \bm{L}^2(\oma;\R^{\dim}),\quad \bm{v}_h|_{\elCR} \in \RTN_{p}(\elCR)\quad\forall\elCR\in\Ta\},
\end{align*}
where  $\RTN_{p}(\elCR) \coloneqq  \calP_{p}(\elCR;\R^\dim) + \calP_{p}(\elCR)\bm{x}$ denotes the Raviart--Thomas--N\'ed\'elec space of order $p$ on $\elCR$.
It is important to notice that whereas the patch $\oma$ is subordinate to the vertices of the mesh $\calT^n$, the spaces $\calP_{p}(\Ta)$ and $\RTN_p(\Ta)$ are subordinate to the submesh $\Ta$; of course, in the absence of coarsening, this distinction vanishes.

We now introduce the local spatial mixed finite element spaces $\Va$ and $\Qa$, defined by
\begin{align*}
\Va & \coloneqq
\begin{cases}
    \left\{\bm{v}_h \in \Hdivoma\cap \RTNa ,\; \bm{v}_h\cdot \bm{n} =0\text{ on }\p\oma \right\} & \text{if }\ta\in\calVhint,\\
    \left\{\bm{v}_h \in \Hdivoma\cap \RTNa  ,\; \bm{v}_h\cdot \bm{n} =0\text{ on }\p\oma\setminus\DO \right\}& \text{if }\ta\in\calVhext,
\end{cases}
 \\
\Qa & \coloneqq
\begin{cases}
     \left\{ q_h\in \Pa ,\quad (q_h,1)_\oma = 0\right\} & \hspace{4.13cm}\text{if }\ta\in\calVhint,\\
     \; \Pa  & \hspace{4.13cm}\text{if }\ta\in\calVhext.
\end{cases}
\end{align*}
We then define the following space-time mixed finite element spaces
\begin{equation}\label{eq:spacetime_mixed_space_def}
\begin{aligned}
\Vthan \coloneqq \calQ_{q_n}(I_n;\Va), & & & \Qthan \coloneqq \calQ_{q_n}(I_n;\Qa).
\end{aligned}
\end{equation}

\subsection{Data approximation}\label{sec:data_approximation}
Our a posteriori error estimates given in section~\ref{sec:Y_aposteriori} involve certain approximations of the source term $f$ appearing in \eqref{eq:parabolic}. It is helpful to define these approximations here.
First, we define the semi-discrete approximation $f_{\tau}$ of $f$ by $L^2$-orthogonal projection in time. In particular, the approximation $f_{\tau} \in \calQ_{q_n}(I_n;L^2(\Om))$ is defined on each interval $I_n$ by $\int_{I_n}  (f - f_{\tau}, v)\, \dd t = 0$ for all $v\in \calQ_{q_n}(I_n;L^2(\Om))$.
Next, for each $1\leq n \leq N$ and for each $\ta\in\calVh$, let $\Pia$ be the $L^2_{\psia}$-orthogonal projection from $L^2(I_n;L^2_{\psia}(\oma))$ onto $\calQ_{q_n}(I_n;\Pal)$, where $L^2_{\psia}(\oma)$ is the space of measurable functions $v$ on $\oma$ such that $\int_{\oma} \psia \abs{v}^2\,\dd x<\infty$.
In other words, the projection operator $\Pia$ is defined by $\int_{I_n} (\psia \Pia v, q_{\th})_{\oma}\,\dd t = \int_{I_n} (\psia v , q_{\th})_{\oma}\,\dd t$ for all $\, q_{\th} \in \calQ_{q_n}(I_n;\Pal)$.
We adopt the convention that $\Pia v$ is extended by zero from $\oma\times I_n$ to $\Omega\times(0,T)$ for all $v\in L^2(I_n;L^2_{\psia}(\oma))$.
Then, we define $f_{\th}$ by
\begin{equation}\label{eq:f_discrete_approx}
f_{\th} \coloneqq \sum_{n=1}^N\sum_{\ta\in\calVh}\psia \,\Pia f.
\end{equation}

\begin{remark}[Definition of $f_{\th}$]
The somewhat technical appearance of the definition of~$f_{\th}$ is due to the possible variation in polynomial degrees across the mesh and the particular requirements of the analysis of efficiency, in particular the hypotheses of Lemma~\ref{lem:p_robust_lifting} below. Nevertheless, $f_{\th}$ has several important approximation properties. First, for any $1\leq n \leq N$, any $\elCR \in \CR$ and any real-valued polynomial $\phi$ of degree at most $q_n$, we have
\begin{equation}\label{eq:fth_mean_value_zero}
\int_{I_n}(f-f_{\th},1)_{\elCR}\phi\,\dd t = \sum_{\ta\in\calV_K} \int_{I_n}(\psia (f - \Pia f),\phi 1)_{\elCR}\,\dd t = 0,
\end{equation}
where $\calV_K$ denotes the set of vertices of $K$, and where we use the fact that the hat functions $\{\psia\}_{\ta\in\calVh}$ form a partition of unity on $\Omega$. Furthermore, using the orthogonality of the projector $\Pia$ and the fact that $0\leq \psia \leq 1$ in $\Omega$, it is straightforward to show that
\[
\norm{f-f_{\th}}_{L^2(I_n;L^2(\elCR))} \leq \sqrt{d+1} \inf_{w_{\th}\in \calQ_{q_n}(I_n;\calP_{p_{\elCR}}(\elCR))} \,\norm{f- w_{\th}}_{L^2(I_n;L^2(\elCR))},
\]
This shows that $f_{\th}$ defines an approximation of $f$ that is at least of the same order as the one associated with the finite element approximation.
\end{remark}

\subsection{Flux reconstruction}\label{sec:flux_reconstruction_def}
For each $1\leq n \leq N$ and each $\ta\in\calVh$, let the scalar function $\gtautha \in \calQ_{q_n}(I_n;\Pa) $ and vector field $\tautha \in \calQ_{q_n}(I_n;\RTNa) $ be defined by
\begin{subequations}\label{eq:tau_g_def}
\begin{align}
\tautha &\coloneqq - \psia \nabla u_{\th}|_{\oma\times I_n},\label{eq:tau_def} \\
\gtautha &\coloneqq \psia\,\left(\Pia  f  - \p_t \Uth\right)|_{\oma\times I_n} - \nabla \psia \cdot \nabla u_{\th}|_{\oma\times I_n}.\label{eq:g_def}
\end{align}
\end{subequations}
We claim that for all $\ta\in \calVhint$,
\begin{equation}
\label{eq:gthan_mean_value_zero}
\begin{aligned}
(\gtautha(t),1)_{\oma} = 0 & & & \forall\,t\in I_n,
\end{aligned}
\end{equation}
which is equivalent to showing that $\gtautha \in \Qthan$ for all $\ta\in\calVh$.
Indeed, we first observe that the construction of the numerical scheme, in particular identity~\eqref{eq:num_scheme_equiv}, implies that, for any univariate real-valued polynomial $\phi$ of degree at most ${q_n}$ on $I_n$,
\[
\int_{I_n} (\gtautha, \phi 1)_{\oma}\,\dd t =
\int_{I_n}  \big(f, \phi\, \psia\big)_{\oma} -  \big(\p_t \Uth, \phi\,\psia\big)_{\oma} - \big(\nabla u_{\th}, \nabla (\phi \, \psia )\big)_{\oma} \,\dd t =0,
\]
where we have used the orthogonality of the projection~$\Pia$ and the fact that $\phi \psia \in \calQ_{{q_n}}(I_n;\Vn)$ is a valid test function in \eqref{eq:num_scheme_equiv}. Since the function $\gtautha$ is polynomial in time with degree at most $q_n$, i.e.\ $\gtautha \in \calQ_{q_n}(I_n;\Pa)$, we deduce \eqref{eq:gthan_mean_value_zero}.

\begin{definition}\label{def:flux_construction_1}
Let $u_\th\in \Vt$ be the numerical solution of~\eqref{eq:num_scheme}. For each time-step interval $I_n$ and for each vertex $\ta\in\calVh$, let the space-time mixed finite element spaces $\Vthan$ and $\Qthan$ be defined by \eqref{eq:spacetime_mixed_space_def}.
Let $\gtautha$ and $\tautha$ be defined by \eqref{eq:tau_g_def}.
Let $\stha \in \Vthan$ be defined by
\begin{equation}\label{eq:stha_minimization_def}
\stha \coloneqq \argmin_{\substack{ \bm{v}_h \in \Vthan \\ \nabla{\cdot} \bm{v}_h = \gtautha}}\int_{I_n} \norm{\bm{v}_h - \tautha}_{\oma}^2\,\dd t.
\end{equation}
Then, after extending $\stha$ by zero from $\oma\times I_n$ to $\Om \times (0,T)$ for each $\ta \in \calVh$ and for each $1\leq n\leq N$, we define
\begin{equation}\label{eq:flux_reconstruction_1}
\sth\coloneqq \sum_{n=1}^N \sum_{\ta \in \calVh} \stha.
\end{equation}
\end{definition}
Note that $\stha\in \Vthan$ is well-defined for all $\ta\in\calVh$: in particular, for interior vertices $\ta\in\calVhint$, we use~\eqref{eq:gthan_mean_value_zero} to guarantee the compatibility of the datum $\gtautha$ with the constraint $\nabla{\cdot}\stha = \gtautha$.

The following key result shows that $\sth$ from Definition~\ref{def:flux_construction_1} leads to an equilibrated flux.
\begin{theorem}[Equilibration]\label{thm:sigma_th_equilibration}
Let the flux reconstruction~$\sth$ be defined by~\eqref{eq:flux_reconstruction_1} of Definition~{\upshape\ref{def:flux_construction_1}}.
Then $\sth \in L^2(0,T;\Hdiv)$ and we have \eqref{eq:sigma_th_equilibration}, where the discrete approximation $f_{\th}$ is defined in \eqref{eq:f_discrete_approx}.
\end{theorem}
\begin{proof}
After extending each $\stha$ by zero from $\oma\times I_n$ to $\Om \times (0,T)$, we have $\stha \in L^2(0,T;\Hdiv)$ as a consequence of the boundary conditions included in the definition of the space $\Va$. This immediately implies that $  \sth \in L^2(0,T;\Hdiv)$.
To show~\eqref{eq:sigma_th_equilibration}, the definition of the flux reconstruction~$\sth$ in \eqref{eq:flux_reconstruction_1} implies that for any time-step interval $I_n $ and any $K\in\calT^n$,
\begin{equation}\label{eq:lem_flux_properties_2}
\begin{split}
\nabla{\cdot} \sth|_{K\times I_n} & = \sum_{\ta\in\calV_K} \nabla{\cdot} \stha|_{K\times I_n} = \sum_{\ta\in\calV_K} \gtautha|_{K\times I_n}
\\ & =\sum_{\ta\in \calV_K} \big( \psia \Pia f -  \psia  \p_t\Uth  - \nabla \psia \cdot \nabla u_{\th}\big)|_{K\times I_n}
\\ &= (f_{\th}- \p_t\Uth)|_{K\times I_n},
\end{split}
\end{equation}
where $\calV_K$ denotes the set of vertices of $K$, where we use the fact that the hat functions $\{\psia\}_{\ta\in\calVh}$ form a partition of unity in order to pass to the last line of \eqref{eq:lem_flux_properties_2}, and where we have used the definition of $f_{\th}$ in \eqref{eq:f_discrete_approx}. This yields \eqref{eq:sigma_th_equilibration} as required.
\qquad\end{proof}

For the purposes of practical implementation, it is easily seen that, for each time-step interval $I_n$, the fluxes $\stha$ can be computed by solving $q_n+1$ independent spatial mixed finite element problems, provided only that an orthogonal or orthonormal polynomial basis is used in time over $I_n$. Moreover, the $q_n+1$ linear systems each share the same matrix, which helps to simplify the implementation and reduce the computational cost.
\begin{lemma}[Decoupling]\label{lem:decoupling}
Let $\stha \in \Vthan$ be defined by~\eqref{eq:stha_minimization_def}. Then $\stha$ is equivalently uniquely defined by: let $(\stha,\rtha)\in \Vthan\times \Qthan$ solve
\begin{subequations}\label{eq:flux_construction_equilibrium}
\begin{align}
&\int_{I_n}(\stha,\bm{v}_{\th})_\oma - (\nabla{\cdot} \bm{v}_{\th}, \rtha)_\oma \,\dd t  = \int_{I_n} (\tautha,\bm{v}_{\th})_\oma \,\dd t & & \forall\, \bm{v}_{\th} \in \Vthan,\\
&\int_{I_n} (\nabla{\cdot} \stha,q_{\th})_\oma \,\dd t = \int_{I_n} (\gtautha,q_{\th})_\oma \,\dd t & & \forall \,q_{\th}\in\Qthan,\label{eq:flux_construction_equilibrium_divergence_constraint}
\end{align}
\end{subequations}
Furthermore, for each $1\leq n \leq N$, let $\{\phi_j^n\}_{j=0}^{q_n}$ be an $L^2(I_n)$-orthonormal basis for the space of univariate real-valued polynomials of degree at most~$q_n$. For each $\ta\in\calVh$, define the functions $\{g^{\ta,n}_{h,j}\}_{j=0}^{q_n}$ and $\{ \tauthj\}_{j=0}^{q_n}$ over the patch $\oma$ by
\begin{equation}\label{eq:decoupling_1}
\begin{aligned}
g^{\ta,n}_{h,j} \coloneqq  \int_{I_n} \gtautha \phi^n_j\,\dd t ,
& & & \tauthj \coloneqq \int_{I_n} \tautha \phi^n_j\,\dd t .
\end{aligned}
\end{equation}
Then, the solution $(\stha,\rtha)$ of
\eqref{eq:flux_construction_equilibrium} can be obtained by solving the following spatial problems: for each $0 \leq j \leq {q_n}$, find $\sthj \in \Va$ and $r ^{\ta,n}_{h,j}$ in $\Qa$ such that
\begin{subequations}\label{eq:decoupling_2}
\begin{align}
&(\sthj,\bm{v}_{h})_\oma - (\nabla{\cdot} \bm{v}_{h}, \rthj)_\oma   =  (\tauthj,\bm{v}_h)_\oma  & & \forall\, \bm{v}_{h} \in \Va,\\
&(\nabla{\cdot} \sthj,q_{h})_\oma  =  (g^{\ta,n}_{h,j},q_{h})_\oma  & & \forall \,q_{h}\in\Qa,
\end{align}
\end{subequations}
and then by defining $ \stha \coloneqq \sum_{j=0}^{q_n} \sthj \phi_j^n$ and $\rtha \coloneqq \sum_{j=0}^{q_n} \rthj \phi_j^n$.
\end{lemma}

\begin{remark}
The analysis in the subsequent sections shows that one particular advantage of the equilibrated flux~$\sth$ of Definition~{\upshape\ref{def:flux_construction_1}} is that it leads to estimators that are robust with respect to coarsening (and refinement) between time-steps. The price to pay is that the size of the linear systems in~\eqref{eq:decoupling_2} grows with the size of coarsening between two successive time-steps, as~\eqref{eq:decoupling_2} are defined on the patches $\oma$ partitioned by the common refinement mesh $\CR$ for each $1\leq n \leq N$. The analysis in~{\upshape\cite[Section~6]{ErnSmearsVohralik2016c}}, though, shows that this computational cost can be significantly reduced to the solution of two low-order systems over the patches $\oma$, followed by local high-order corrections on the sub-patches of $\Ta$. We refer the reader to {\upshape\cite[Section~6]{ErnSmearsVohralik2016c}} for the full details of this approach.
\end{remark}

\section{Main results}\label{sec:Y_aposteriori}
In this section, we present the a posteriori error estimate featuring guaranteed upper bounds, local space-time efficiency, and polynomial-degree robustness.
Let the norm $\norm{\cdot}_{\nEy}\colon Y+\Vt \tends \R_{\geq 0}$ be defined by
\begin{equation}\label{eq:EY_norm_def}
\begin{aligned}
\norm{v}_{\nEy}^2 \coloneqq \norm{\calI v}_Y^2 + \norm{v-\calI v}_X^2  & & & \forall \,v\in Y+\Vt,
\end{aligned}
\end{equation}
where we recall Remark~\ref{rem:radau_extension} on the extension of the linear operator $\calI$ to $Y+\Vt$.
Since the exact solution $u\in Y$ implies that $\calI u = u$, we have the identities
\begin{equation}\label{eq:error_measure_Y}
\begin{split}
\Ey^2 & = \norm{u-\Uth}_{Y}^2 + \norm{ u_{\th}-\Uth}_X^2 \\
 & =  \norm{u-\Uth}_{Y}^2 + \sum_{n=1}^{N} \tfrac{\tau_n ({q_n}+1)}{(2q_n+1)\,(2q_n+3)}\norm{\nabla \lld u_{\th} \rrd_{n-1} }^2,
\end{split}
\end{equation}
where we have simplified $\int_{I_n} \norm{\nabla(u_{\th}-\Uth)}^2\,\dd t=\tfrac{\tau_n ({q_n}+1)}{(2q_n+1)\,(2q_n+3)}\norm{\nabla \lld u_{\th} \rrd_{n-1} }^2$, which is an identity easily deduced from \eqref{eq:def_radau_reconstruction} and from $\int_{I_n}\abs{L_{q}^n}^2\,\dd t=\tfrac{\tau_n}{2q+1}$ for all $q\geq 0$; see also \cite{Schotzau2010}.
We also introduce the localized seminorms $\abs{\cdot}_{\nEya}$, for each $1\leq n \leq N$ and each $\ta\in \calVh$, defined by
\begin{equation}\label{eq:EY_norm_def_localized}
\abs{v}_{\nEya}^2 \coloneqq\int_{I_n}\norm{\p_t\, \calI v}_{H^{-1}(\oma)}^2 + \norm{\nabla\calI v}_{\oma}^2+\norm{\nabla(v-\calI v)}_{\oma}^2\,\dd t \quad \forall\,v\in Y+\Vt.
\end{equation}
Similarly to \eqref{eq:error_measure_Y}, we find that
\begin{multline}\label{eq:error_measure_Y_localized}
\Eya^2 = \int_{I_n} \norm{\p_t(u-\Uth)}_{H^{-1}(\oma)}^2 + \norm{\nabla (u-\Uth)}_{\oma}^2 \,\dd t \\
+ \tfrac{\tau_n ({q_n}+1)}{(2q_n+1)\,(2q_n+3)} \norm{\nabla \lld u_{\th} \rrd_{n-1} }_{\oma}^2.
\end{multline}

Although it might not be immediately obvious that $\Ey$ is equivalent to the Hilbertian sum of the $\Eya$, up to data oscillation, this will come as a consequence of the results shown here and in section~\ref{sec:efficiency_Y_norm}.
We are now ready to state our main results in Theorems~\ref{thm:Y_norm_relation} and~\ref{thm:Y_norm_guaranteed_efficiency} below. It is helpful to denote the time-localized dual norm of the residual by
\begin{equation}\label{eq:time_localized_residual_norm}
\norm{\calR_Y(\Uth)|_{I_n}}_{X^\prime}\coloneqq\sup_{v\in X,\;\norm{v}_X=1}\int_{I_n}(f,v)-\pair{\p_t \Uth}{v}-(\nabla\Uth,\nabla v)\,\dd t.
\end{equation}
Note that $\norm{\calR_Y(\Uth)|_{I_n}}_{X^\prime}$ can always be bounded from above by the restriction of the $Y$-norm of the error $u-\Uth$ to the time-step interval~$I_n$.

\begin{theorem}[Equivalence of norms]\label{thm:Y_norm_relation}
Let the norm $\NEy$ be defined by \eqref{eq:EY_norm_def}, and, for each $1\leq n \leq N$, let the temporal data oscillation $\etaOscTn$ and the coarsening error indicator $\etaCJ$ be defined by
\begin{subequations}
\begin{align}
\etaCJ  &\coloneqq \sqrt{\tfrac{\tau_n ({q_n}+1)}{(2q_n+1)\,(2q_n+3)}} \norm{ \nabla \left\{  u_{\th}(t_{n-1}) - P_h^n [ u_{\th}(t_{n-1}) ] \right\} }, \label{eq:etaCJ_def}
\\ [\etaOscTn]^2 &\coloneqq \int_{I_n}\norm{f(t)- f_{\tau}(t)}_{H^{-1}(\Om)}^2\,\dd t , \label{eq:etaOscT_def}
\end{align}
\end{subequations}
where $P_h^{n}\colon H^1_0(\Om)\tends \Vn$ denotes the elliptic orthogonal projection onto $\Vn$ defined by $(\nabla P_h^n w , \nabla v_h) = (\nabla w, \nabla v_h)$ for all $v_h\in \Vn$.
Then, we have
\begin{equation}\label{eq:main_jump_bound}
\int_{I_n}\norm{\nabla(u_{\th}-\Uth)}^2\dd t \leq 8 \norm{\calR_Y(\Uth)|_{I_n}}_{X^{\prime}}^2 + \min\left\{ [\etaCJ]^2, 8 [\etaOscTn]^2\right\},
\end{equation}
where $\norm{\calR_Y(\Uth)|_{I_n}}_{X^\prime}$ is defined in \eqref{eq:time_localized_residual_norm}. Furthermore, we have
\begin{equation}\label{eq:norm_measure_equivalence}
\norm{u-\Uth}_{Y}^2 \leq \Ey^2 \leq 9 \norm{u-\Uth}_{Y}^2 + \sum_{n=1}^N \min\left\{ [\etaCJ]^2, 8 [\etaOscTn]^2 \right\}.
\end{equation}
\end{theorem}
We delay the proof of Theorem~\ref{thm:Y_norm_relation} until section~\ref{sec:Y_norm_relation} below.

\begin{remark}[Equivalence]
Theorem~{\upshape\ref{thm:Y_norm_relation}} shows that $\Ey$ and $\norm{u-\Uth}_{Y}$ are globally equivalent up to the minimum of temporal data oscillation and coarsening errors. In particular, one of our key contributions here is to obtain polynomial-degree independent constants in~\eqref{eq:norm_measure_equivalence}.  It is important to note that although $\Ey$ and $\norm{u-\Uth}_{Y}$ are essentially globally equivalent, their local distributions may differ.
\end{remark}

\begin{remark}[Relation to \cite{Verfurth2003}]
A similar result to \eqref{eq:main_jump_bound} was previously obtained in the lowest-order case $q_n=0$ by Verf\"urth~{\upshape\cite{Verfurth2003}}; see in particular the bounds of {\upshape\cite[Section~7]{Verfurth2003}} for what is denoted there $\tfrac{\tau_n}{3} \abs{u_h^n-u_h^{n-1}}_{1}^2$, which is equivalent to $\int_{I_n}\norm{\nabla(u_{\th}-\Uth)}^2\,\dd t$ with $q_n=0$ in our notation.
For higher polynomial degrees, we note that Gaspoz, Kreuzer, Siebert and Ziegler~{\upshape\cite{GaspozKreuzerSiebertZiegler2016}} have obtained independently an inequality of a similar kind as~\eqref{eq:main_jump_bound}.
\end{remark}

We introduce the following a posteriori error estimators and data oscillation terms:
\begin{subequations}\label{eq:estimators}
\begin{align}
\etaEq (t) &\coloneqq  \norm{\sth(t) + \nabla \Uth(t)}_{K} , \label{eq:etaEq_def} \\
\etaJ & \coloneqq \sqrt{\tfrac{\tau_n ({q_n}+1)}{(2q_n+1)(2q_n+3)}}\, \norm{\nabla \lld u_{\th} \rrd_{n-1}}_K, \label{eq:etaJ_def}\\
\etaOscS (t) & \coloneqq \Biggr[\sum_{\elCR\in \CR,\; \elCR \subset K} \frac{h_{\elCR}^2 }{\pi^2} \norm{ f_\tau(t) - f_{\th}(t)}_{\elCR}^2\Biggr]^{\frac{1}{2}}, \label{eq:etaOscS_def} \\
\etaOscT (t) &\coloneqq \norm{f(t)-f_{\tau}(t)}_{H^{-1}(\Om)}, \\
\etaOscInit & \coloneqq \norm{u_0-\Pi_h u_0},
\end{align}
\end{subequations}
where $t\in I_n$, $K\in \calT^n$, the equilibrated flux~$\sth$ is defined in Definition~\ref{def:flux_construction_1}, and where the data approximations $f_{\tau}$ and $f_{\th}$ are respectively defined in section~\ref{sec:data_approximation}.
The two estimators $\etaEq$ and $\etaJ$ are our principal estimators, where $\etaEq$ measures respectively the lack of $\bm{H}(\Div)$-conformity of the gradient of the reconstructed solution $\Uth$, and where $\etaJ$ measures the lack of temporal conformity of the numerical solution~$u_{\th}$.
The term $\etaOscS$ represents the data oscillation due to the spatial discretisation, whereas $\etaOscT$ represents the data oscillation due to the temporal discretisation.
We define the global a posteriori error estimators as
\begin{subequations}\label{eq:etaY_defs}
\begin{align}
\eta_{Y}^2 &\coloneqq \sum_{n=1}^{N} \int\limits_{I_n}\biggr[ \bigg\{\sum_{K\in\calT^n}  [\etaEq + \etaOscS]^2  \bigg\}^{\frac{1}{2}}\hspace{-0.5ex} + \etaOscT  \biggr]^2 \hspace{-0.5ex} \dd t  + [\etaOscInit]^2, \label{eq:tetaY_def} \\
\eta_{\nEy}^2 &\coloneqq \eta_Y^2 + \sum_{n=1}^N\sum_{K\in\calT^n}[\etaJ]^2.
\label{eq:etaY_def}
\end{align}
\end{subequations}
Notice that in the absence of data oscillation, namely if $f=f_{\tau}=f_{\th}$ and $u_0=\Pi_h u_0$, then $\eta_Y$ simplifies to $\eta_Y^2 =\int_0^T \norm{\sth + \nabla \Uth}^2\,\dd t$, and $\eta_{\nEy}$ simplifies to $\eta_{\nEy}^2=\int_0^T\norm{\sth+\nabla \Uth}^2+\norm{\nabla(u_{\th}-\Uth)}^2\,\dd t$.

Recall that we write $a\lesssim b$ for two quantities $a$ and $b$ if $a \leq C b$ with a constant $C$ depending only on the shape regularity of $\calT^n$ and $\CR$, but otherwise independent of the mesh-size, time-step size, and polynomial degrees in space and time.

\begin{theorem}[$\nEy$-norm a posteriori error estimate]\label{thm:Y_norm_guaranteed_efficiency}
Let $u \in Y$ be the weak solution of \eqref{eq:parabolic}, let $u_{\th}\in \Vt$ denote the solution of the numerical scheme~\eqref{eq:num_scheme}, and let $\Uth$ denote its temporal reconstruction, where the operator $\calI$ is defined in~\eqref{eq:def_radau_reconstruction}.
Let $\sth$ denote the equilibrated flux of Definition~{\upshape\ref{def:flux_construction_1}}. Let $\norm{\cdot}_{\nEy}$ be defined in~\eqref{eq:EY_norm_def}, and let the a posteriori error estimators be defined in~\eqref{eq:estimators}, with $\eta_{\nEy}$ defined in~\eqref{eq:etaY_defs}. Then, we have the guaranteed upper bound
\begin{equation}\label{eq:Y_norm_upper_bounds_1}
 \Ey  \leq \eta_{\nEy}.
\end{equation}
Moreover, for each $1\leq n \leq N$ and for each $K\in \calT^n$, the indicators satisfy the following local efficiency bound:
\begin{equation}\label{eq:Y_norm_guaranteed_efficiency_lower_local}
\int_{I_n} [\etaEq]^2 \,\dd t + [\etaJ]^2  \lesssim \sum_{\ta \in \calV_K}\left\{ \Eya^2 +  [\etaOsca]^2 \right\},
\end{equation}
where $\abs{\cdot}_{\nEya}$ is defined in \eqref{eq:EY_norm_def_localized}, $\calV_K$ is the set of vertices of the element $K$, and the local data oscillation term $\etaOsca$ is defined by
\begin{equation}\label{eq:patch_data_oscillation}
[\etaOsca]^2 \coloneqq \int_{I_n} \norm{f-\Pia f}_{H^{-1}(\oma)}^2\,\dd t.
\end{equation}
Furthermore, we have the following global efficiency bound for $\Ey$:
\begin{equation}\label{eq:Y_norm_guaranteed_efficiency_lower_global}
\sum_{n=1}^N  \sum_{K\in\calT^n} \left[ \int_{I_n} [\etaEq]^2 \,\dd t + [\etaJ]^2  \right] \lesssim \Ey^2 + \sum_{n=1}^N \sum_{\ta\in\calVh} [\etaOsca]^2 .
\end{equation}
\end{theorem}
The proof of Theorem~\ref{thm:Y_norm_guaranteed_efficiency} is postponed to the following sections: the proof of the upper bound~\eqref{eq:Y_norm_upper_bounds_1} is given in section~\ref{sec:proof_upper}, and the proof of the bounds~\eqref{eq:Y_norm_guaranteed_efficiency_lower_local}~and~\eqref{eq:Y_norm_guaranteed_efficiency_lower_global} is the subject of section~\ref{sec:efficiency_Y_norm}.
Theorem~{\upshape\ref{thm:Y_norm_guaranteed_efficiency}} shows the local space-time efficiency of the estimators with respect to $\Ey$.
As a consequence of the proof of Theorem~\ref{thm:Y_norm_guaranteed_efficiency}, we can also show guaranteed upper bounds and local-in-time and global-in-space efficiency of the estimators with respect to $\norm{u-\Uth}_Y$, thereby generalising the results to {\upshape\cite{Verfurth2003}} to higher-order approximations, see Corollary~{\upshape\ref{cor:space_global_efficiency}} below.
\begin{corollary}[$Y$-norm a posteriori error estimate]\label{cor:space_global_efficiency}
Let the estimator $\eta_Y$ be defined by \eqref{eq:tetaY_def}. Then, we have
\begin{equation}\label{eq:Y_norm_upper_bounds_2}
\norm{u-\Uth}_{Y} \leq \eta_Y.
\end{equation}
Furthermore, for each $1\leq n \leq N$, we have
\begin{multline}\label{eq:Y_norm_alternative_efficiency}
\sum_{K\in\calT^n} \left[ \int_{I_n} [\etaEq]^2 \,\dd t + [\etaJ]^2  \right] \lesssim \int_{I_n} \norm{\p_t(u-\Uth)}_{H^{-1}(\Om)}^2+\norm{\nabla(u-\Uth)}^2 \,\dd t  \\ + \min\left\{ [\etaCJ]^2, 8 [\etaOscTn]^2\right\} + \sum_{\ta\in\calVh}  [\etaOsca]^2 .
\end{multline}
\end{corollary}

\begin{remark}[Temporal data oscillation]
The temporal data oscillation term~$\etaOscT$ is defined with respect to a negative norm, as usual in the literature~{\upshape\cite{ErnVohralik2010,Verfurth2003}}. Similarly to {\upshape\cite{ErnVohralik2010,Verfurth2003}}, this temporal data oscillation term can be of the same order as the error in terms of the time-step size. Since this term already appears in the upper bounds of the residual-based estimates of {\upshape\cite[Eq.~(1.5)]{Verfurth2003}}, it is seen that this issue is not related to the choice of equilibrated flux a posteriori error estimators, but is rather a part of the error estimation in the $Y$-norm. In practical computations, it is often advisable to determine a minimal temporal resolution for reducing this term to within a prescribed tolerance, in advance of solving the numerical scheme \eqref{eq:num_scheme}. Although the negative norm appearing in the definition of $\etaOscT$ is non-computable, there are several possibilities for estimating it. First, we mention that $\etaOscT$ is bounded from above by $C_{\Om}\norm{f-f_{\tau}}$, with $C_{\Om}$ the constant of the global Poincar\'e inequality, although this can be pessimistic in practice. If $f$ is a finite tensorial product of spatial and temporal functions, then sharper bounds can be obtained by solving a set of independent coarse and low-order conforming approximations for elliptic problems, followed by equilibrated flux a posteriori error estimates to achieve guaranteed upper bounds. Finally, we also mention that this issue motivates a posteriori error estimators in other norms: in particular, we show in {\upshape\cite{ErnSmearsVohralik2016b}} that $X$-norm a posteriori estimates benefit from data oscillation terms that are of higher-order by an additional factor of $\sqrt{\tau}+h$.
\end{remark}

\section{Proof of equivalence between $\Ey$ and $\norm{u-\Uth}_Y$}\label{sec:Y_norm_relation}
In this section, we prove Theorem~\ref{thm:Y_norm_relation}, along with some corollary results, which relate $\Ey$ with $\norm{u-\Uth}_Y$.
Our starting point involves the following two original bounds on the norms of the jumps, which generalize one of the key results of Verf\"urth~\cite{Verfurth2003} for the lowest-order case $q_n=0$. In fact, our result sharpens and simplifies the proof of the result of \cite{Verfurth2003} even in the lowest-order case.

\begin{lemma}\label{lem:jump_lower_bound}
For each $1\leq n \leq N$, let $P_h^{n}\colon H^1_0(\Om)\tends \Vn$ denote the elliptic orthogonal projection to $\Vn$ defined by $(\nabla P_h^n w , \nabla v_h) = (\nabla w, \nabla v_h) $ for all $v_h\in \Vn$. Then, for each $1\leq n\leq N$, the jump $\lld u_{\th} \rrd_{n-1}$ satisfies
\begin{multline}\label{eq:jump_lower_bound}
\frac{\tau_n}{8q_n+4} \norm{\nabla \lld u_{\th} \rrd_{n-1} }^2 \leq \norm{\calR_Y(\Uth)|_{I_n}}_{X^\prime}^2 \\ + \frac{\tau_n}{8q_n+4} \norm{\nabla \left\{ u_{\th}(t_{n-1}) - P_h^n [ u_{\th}(t_{n-1})] \right\} }^2,
\end{multline}
where $\norm{\calR_Y(\Uth)|_{I_n}}_{X^\prime}$ is defined in~\eqref{eq:time_localized_residual_norm}.
Furthermore, we also have the alternative bound
\begin{equation}\label{eq:jump_lower_bound_alternative}
\frac{\tau_n}{8 q_n+12}\norm{\nabla \lld u_{\th} \rrd_{n-1} }^2
\leq 2 \left( \norm{\calR_Y(\Uth)|_{I_n}}_{X^\prime}^2
+ [\etaOscTn]^2 \right).
\end{equation}
\end{lemma}
\begin{proof}
First, note that $ (u_{\th} - \Uth)|_{I_n} = \tfrac{(-1)^{q_n}}{2} ( L_{q_n+1}^n-L_{q_n}^n) \lld u_{\th} \rrd_{n-1}$ belongs to the space $\calQ_{q_n+1}(I_n;\VCR)$. We define the test function $v_{\th} \coloneqq - \tfrac{(-1)^{q_n}}{2} L_{q_n}^n P_h^n \lld u_{\th} \rrd_{n-1}$, which belongs to $\calQ_{q_n}(I_n;\Vn) $, and we use it in equation~\eqref{eq:num_scheme_equiv} for the numerical scheme, which yields, by orthogonality of the Legendre polynomials and by the definition of the orthogonal projector $P_h^n$, the identity
\begin{equation}\label{eq:jump_lower_bound_2}
\begin{split}
 \int_{I_n} \norm{\nabla v_{\th}}^2 \,\dd t & = \frac{\tau_n}{8q_n+4} \norm{\nabla P_h^n \lld u_{\th} \rrd_{n-1} }^2
 = \int_{I_n} (\nabla (u_{\th} - \Uth), \nabla v_{\th}) \,\dd t \\
 & = \int_{I_n} ( f-\p_t \Uth,v_{\th})  - (\nabla \Uth,\nabla v_{\th}) \,\dd t .
\end{split}
\end{equation}
Therefore, we have $\int_{I_n} \norm{\nabla v_{\th}}^2 \,\dd t\leq \norm{\calR_Y(\Uth)|_{I_n}}_{X^\prime}^2$.
This bound yields the desired result \eqref{eq:jump_lower_bound} once it is combined with \eqref{eq:jump_lower_bound_2} and the orthogonality relation
\[
\begin{split}
\norm{\nabla P_h^n \lld u_{\th} \rrd_{n-1} } ^2
  & = \norm{\nabla \lld u_{\th} \rrd_{n-1}}^2 - \norm{\nabla \left\{ \lld u_{\th} \rrd_{n-1} - P_h^n  \lld u_{\th} \rrd_{n-1} \right\} }^2 \\
  & = \norm{\nabla \lld u_{\th} \rrd_{n-1}}^2 - \norm{\nabla \left\{ u_{\th}(t_{n-1}) - P_h^n [ u_{\th}(t_{n-1}) ] \right\} }^2,
\end{split}
\]
where the last equality above follows from the facts that $\lld u_{\th} \rrd_{n-1} = u_{\th}(t_{n-1}) - u_{\th}(t_{n-1}^+)$ and that $u_{\th}(t_{n-1}^+) \in \Vn$. This completes the proof of the first bound~\eqref{eq:jump_lower_bound}.

We now turn to the proof of~\eqref{eq:jump_lower_bound_alternative}; the main difference in the proofs of~\eqref{eq:jump_lower_bound} and~\eqref{eq:jump_lower_bound_alternative} is that above we appealed to the numerical scheme using a discrete test function, whereas to establish \eqref{eq:jump_lower_bound_alternative}, we shall now consider a higher-order polynomial function that is not in the discrete test space. We define $v$ on $I_n$ by $v|_{I_n} \coloneqq \frac{(-1)^{q_n}}{2} L_{q_n+1}^n \, \lld u_{\th} \rrd_{n-1}$, and then we extend $v$ by zero outside of $I_n$, so that $v\in X$. Then, by orthogonality of the Legendre polynomial $L_{q_n+1}^n$ to all polynomials of degree at most $q_n$ on $I_n$, we have the identities
$\int_{I_n} (f_\tau , v)\,\dd t = 0$, $\int_{I_n} (\p_t \Uth, v) \,\dd t =0$ and $\int_{I_n} (\nabla u_{\th},\nabla v) \,\dd t = 0$. Therefore, we obtain
\[
\begin{split}
 \int_{I_n} \norm{\nabla v}^2\,\dd t &= \frac{\tau_n}{8q_n+12} \norm{\nabla  \lld u_{\th} \rrd_{n-1} }^2 = \int_{I_n} (\nabla (u_{\th}-\Uth),\nabla v)\,\dd t
\\ & = \int_{I_n} (f,v)-(\p_t \Uth,v)-(\nabla \Uth,\nabla v) + (f_{\tau}-f,v) \,\dd t.
\end{split}
\]
The desired result \eqref{eq:jump_lower_bound_alternative} then follows straightforwardly from the above identity.\qquad\end{proof}

\paragraph{Proof of~Theorem~{\upshape\ref{thm:Y_norm_relation}}}
The first inequality $\norm{u-\Uth}_{Y}^2\leq \Ey^2$ is obvious from the definition of $\Ey$ in~\eqref{eq:error_measure_Y}.
Recalling the definitions of $\etaJ$ in \eqref{eq:etaJ_def} and $\etaCJ$ in \eqref{eq:etaCJ_def}, we deduce from \eqref{eq:jump_lower_bound} and \eqref{eq:jump_lower_bound_alternative} that
\begin{equation}
\sum_{K\in\calT^n}[\etaJ]^2 \leq \frac{4(q_n+1)}{(2q_n+3)} \norm{\calR_Y(\Uth)|_{I_n}}_{X^\prime}^2 + [\etaCJ]^2 \leq 2 \norm{\calR_Y(\Uth)|_{I_n}}_{X^\prime}^2 + [\etaCJ]^2,
\end{equation}
and that
\begin{equation}
\begin{split}
\sum_{K\in\calT^n}[\etaJ]^2
&\leq \frac{8(q_n+1)}{(2q_n+1)} \left( \norm{\calR_Y(\Uth)|_{I_n}}_{X^\prime}^2 +  [\etaOscTn]^2 \right) \\
& \leq 8 \norm{\calR_Y(\Uth)|_{I_n}}_{X^\prime}^2 + 8 [\etaOscTn]^2.
\end{split}
\end{equation}
Therefore, we obtain \eqref{eq:main_jump_bound} by taking the minimum of the right-hand sides of the above bounds.
Finally, we get \eqref{eq:norm_measure_equivalence} by summing the above inequality over all time-steps and noting that $\sum_{n=1}^N  \norm{\calR_Y(\Uth)|_{I_n}}_{X^\prime}^2 = \norm{\calR_Y(\Uth)}_{X^\prime}^2 \leq \norm{u-\Uth}_{Y}^2 $ which follows from \eqref{eq:Y_error_residual_equivalence}. \qquad\endproof

It is possible to obtain slightly sharper variants of Theorem~\ref{thm:Y_norm_relation} under more specific assumptions. For instance, the following corollary shows that $\Ey$ is equivalent to $\norm{u-\Uth}_{Y}$, without any additional data oscillation, whenever the mesh coarsening error is kept relatively small to the jumps.
\begin{corollary}\label{cor:error_measure_equivalence}
Using the notation of Lemma~{\upshape\ref{lem:jump_lower_bound}}, assume that there exists a constant $\theta \in [0,1)$ such that $\norm{\nabla \left[ \lld u_{\th} \rrd_{n-1} - P_h^n  \lld u_{\th} \rrd_{n-1} \right] }^2 \leq \theta \norm{\nabla \lld u_{\th} \rrd_{n-1} }^2$ for each $1\leq n\leq N$.
Then, we have
\begin{equation}\label{eq:error_measure_equivalence_cor}
\norm{u-\Uth}_{Y}^2 \leq \Ey^2 \leq \frac{3-\theta}{1-\theta} \norm{u-\Uth}_{Y}^2.
\end{equation}
\end{corollary}
\begin{proof}
The result is a consequence of Lemma~\ref{lem:jump_lower_bound} and the fact that $\lld u_{\th} \rrd_{n-1} - P_h^n  \lld u_{\th} \rrd_{n-1} = u_{\th}(t_{n-1}) - P_h^n [ u_{\th}(t_{n-1}) ]$ , which leads to $\frac{\tau_n}{8q_n+4} \norm{\nabla \lld u_{\th} \rrd_{n-1} }^2 \leq \tfrac{1}{1-\theta} \norm{\calR_Y(\Uth)|_{I_n}}_{X^\prime}^2$.
Proceeding as in the proof of Theorem~\ref{thm:Y_norm_relation} then yields~\eqref{eq:error_measure_equivalence_cor}.
\end{proof}

\section{Proof of the guaranteed upper bound}\label{sec:proof_upper}
We prove here~\eqref{eq:Y_norm_upper_bounds_1}~and~\eqref{eq:Y_norm_upper_bounds_2}.
First, it is clear from \eqref{eq:error_measure_Y} that  \eqref{eq:Y_norm_upper_bounds_2} immediately implies \eqref{eq:Y_norm_upper_bounds_1}. Therefore, it remains to show \eqref{eq:Y_norm_upper_bounds_2}.
Keeping in mind the equivalence identity~\eqref{eq:Y_error_residual_equivalence} between norms of the errors and residuals, we turn our attention to bounds for the residual norm $\norm{\calR_Y(\Uth)}_{X^\prime} = \sup_{v\in X\setminus\{0\}} \By(u-\Uth,v)/\norm{v}_X$.
To this end, consider an arbitrary function $v\in X$ such that $\norm{v}_X = 1$.
Then, we obtain
\[
\pair{\calR_Y(\Uth)}{v}
= \int_0^T (f-\p_t\Uth - \nabla{\cdot} \sth,v) - (\sth+\nabla\Uth,\nabla v) \,\dd t,
\]
where we have inserted the flux $\sth$ and used integration by parts over $\Om$.
Next, we use~\eqref{eq:sigma_th_equilibration}, and we write $f-f_{\th}=f-f_{\tau} + f_{\tau}-f_{\th}$.
For any $\elCR \in \CR$, $1\leq n \leq N$, we deduce from \eqref{eq:fth_mean_value_zero} that the function $t\mapsto (f_\tau (t) - f_{\th}(t),1)_{\elCR}$, which is a real-valued polynomial of degree at most $q_n$ on $I_n$, vanishes identically on $I_n$.
Therefore, letting $v_{\elCR}(t)$ denote the mean value of $v(t)$ over the element $\elCR\in \CR$, which is defined for a.e.\ $t\in I_n$, we deduce from the Poincar\'e inequality that $\abs{(f_\tau (t) - f_{\th}(t),v(t))_{\elCR} } \leq \frac{ h_{\elCR} }{\pi}  \norm{f_\tau (t) - f_{\th}(t)}_{\elCR} \norm{\nabla v(t)}_{\elCR}$.
Therefore, $\pair{ \calR_Y(\Uth) }{v}$ can be bounded as follows:
\[
\begin{split}
&\pair{\calR_Y(\Uth)}{v}  = \sum_{n=1}^N \int_{I_n}  (f-f_{\th},v) - (\sth+\nabla\Uth,\nabla v) \,\dd t \\
&\leq \sum_{n=1}^N \int\limits_{I_n} \sum_{K\in \calT^n} \etaEq \norm{\nabla v}_K + \Biggr[ \sum_{\elCR\in\CR} \frac{ h_{\elCR} }{\pi}  \norm{f_\tau (t) - f_{\th}(t)}_{\elCR} \norm{\nabla v(t)}_{\elCR} \Biggr] + \etaOscT \norm{\nabla v} \dd t \\
& \leq \sum_{n=1}^N \int\limits_{I_n} \Biggr[\sum_{K\in\calT^n} [\etaEq+\etaOscS]\norm{\nabla v}_K \Biggr]+ \etaOscT \norm{\nabla v} \,\dd t \\
& \leq \sum_{n=1}^N \int\limits_{I_n}\Biggr[ \left\{ \sum_{K\in\calT^n} [\etaEq+\etaOscS]^2 \right\}^{\frac{1}{2}} +  \etaOscT \Biggr] \norm{\nabla v}\, \dd t .
\end{split}
\]
Applying the Cauchy--Schwarz inequality leads to an upper bound for $\norm{ \calR_Y(\Uth) }_{X^\prime}$, which we then combine with the identity \eqref{eq:Y_error_residual_equivalence}  relating errors and residuals to obtain $\norm{u-\Uth}_{Y}\leq \eta_Y$. The corresponding upper bound $\Ey \leq \eta_{\nEy}$ then follows immediately, as explained above.
\qquad\endproof

\section{Proof of local space-time efficiency and robustness}\label{sec:efficiency_Y_norm}
We prove here the bounds~\eqref{eq:Y_norm_guaranteed_efficiency_lower_local}, \eqref{eq:Y_norm_guaranteed_efficiency_lower_global}, and~\eqref{eq:Y_norm_alternative_efficiency}.

\subsection{Preliminary result}
The following lemma is a generalisation of important results on polynomial-degree robustness of equilibrated flux estimates from \cite[Thm.~7]{Braess2009}, in two space dimensions, and \cite[Thm~2.3]{ErnVohralik2016} in three space dimensions.
In particular, Lemma~\ref{lem:p_robust_lifting} comes from \cite[Thm~1.2]{ErnSmearsVohralik2016c} on the existence of a discrete polynomial-degree robust $\bm{H}(\Div)$-lifting of data that are piecewise-polynomials with respect to the submesh $\Ta$. Note that \cite[Thm.~7]{Braess2009} and \cite[Thm~2.3]{ErnVohralik2016} only hold for the case where the data are piecewise-polynomials on the elements $K\in \calT^n$ of the patch $\oma$. This generalisation is crucial for allowing arbitrary refinement and coarsening between time-steps.

\begin{lemma}[Polynomial degree-robust stability bound]\label{lem:p_robust_lifting}
For each $1\leq n\leq N$ and each $\ta\in\calVh$, recall that $\Ta$ denotes the restriction of $\CR$ to $\oma$ and that $\psia \in H^1(\oma)\cap \calP_{1}(\Ta)$ denotes the hat function associated with $\oma$. Let $\Gamma_{\ta} = \{ x\in \p\oma,\; \psia(x)=0\}$. Then, for any $f_h^{\ta,n} \in \Pal$ and any $\bxi_h^{\ta,n}\in \RTNal$, where it is further supposed that $(f_h^{\ta},\psia)_{\oma}=(\bxi_h^{\ta,n},\nabla \psia)_{\oma}$ if $\Gamma_{\ta} = \p\oma$, we have
\[
\min_{ \substack{\bm{v}_h \in \Hdivoma \cap \RTNa \\ \nabla{\cdot}\bm{v}_h = \psia f_h^{\ta,n} - \nabla \psia \cdot \bxi_h^{\ta,n} \\ \bm{v}_h\cdot\bm{n}=0 \text{ on }\Gamma_{\ta}  }} \norm{\bm{v}_h + \psia \bxi_h^{\ta,n}}_{\oma} \lesssim \sup_{\vphi \in H^1_0(\oma)\setminus\{0\}} \frac{(f_h^{\ta,n},\vphi)_{\oma}-(\bxi_h^{\ta,n},\nabla \vphi)_{\oma} }{\norm{\nabla \vphi}_{\oma}}
\]
\end{lemma}
\begin{proof}
The result is directly obtained by applying \cite[Thm~1.2]{ErnSmearsVohralik2016c}, where $\Om$ there stands for $\oma$ here, where $\calT$ there stands for $\Ta$ here, and where $\psi_{\dagger}$ there stands for $\psia$ here. In applying \cite[Thm~1.2]{ErnSmearsVohralik2016c}, we use the fact that $\Gamma_{\ta}$ is the union of the faces of the mesh $\Ta$ on which $\psia$ vanishes, and we have simplified the constant appearing there by using the fact that $h_{\oma} \norm{\nabla \psia}_{\infty} \lesssim \norm{\psia}_{\infty}=1$ by shape-regularity.~\qquad\end{proof}

\subsection{Stability of the space-time flux equilibration}
For each $1\leq n\leq N$ and each $\ta \in \calVh$, we introduce the patch residual functional $\Ra$, with $\Ra \colon L^2(I_n,H^1_0(\oma))\tends\R$ defined by
\begin{equation}\label{eq:patch_residual}
\begin{aligned}
\pair{\Ra}{v}  =\int_{I_n} \big(\Pia f - \p_t (\Uth) ,v\big)_{\oma}  -\big(\nabla u_{\th},\nabla v \big)_\oma\dd t,
\end{aligned}
\end{equation}
for all $v\in L^2(I_n,H^1_0(\oma))$.
We are now ready to state the essential result that forms the starting point for our analysis of the efficiency of the error estimators.

\begin{lemma}[Space-time stability bound]\label{lem:flux_reconstruction_stability}
Let $\stha$ denote the patch-wise flux reconstructions of Definition~{\upshape\ref{def:flux_construction_1}}, and let $\Ra$ denote the local patch residual defined by \eqref{eq:patch_residual}. Then, we have
\begin{equation}\label{eq:flux_reconstruction_stability}
\left(\int_{I_n}\norm{\stha + \psia \nabla u_{\th}}_\oma^2\,\dd t\right)^{\frac{1}{2}}  \lesssim \sup_{v \in \calQ_{q_n}(I_n;H^1_0(\oma))\setminus\{0\}} \frac{\pair{\Ra}{v}}{\left(\int_{I_n}\norm{\nabla v}_{\oma}^2\,\dd t\right)^{\frac{1}{2}}},
\end{equation}
where $\calQ_{q_n}(I_n;H^1_0(\oma))$ denotes the space of $H^1_0(\oma)$-valued univariate polynomials of degree at most $q_n$ on $I_n$.
\end{lemma}
\begin{proof}
The definition of $\stha\in\Vthan$ in \eqref{eq:stha_minimization_def} implies that it is enough to show that there exists a $\bm{v}_h \in \Vthan$ such that $\nabla{\cdot}\bm{v}_h = \gtautha$ and such that $\int_{I_n}\norm{\bm{v}_h - \tautha}_{\oma}^2\,\dd t$ is bounded by the right-hand side of \eqref{eq:flux_reconstruction_stability}.
Let $\{\phi_j^n\}_{j=0}^{q_n}$ be an $L^2$-orthonormal basis of polynomials on $I_n$, and let the functions $\{\fthj\}_{j=0}^{q_n}  $ and $\{\bxithj\}_{j=0}^{q_n}$ be defined by
\begin{equation}\label{eq:fthj_bxithj_def}
\begin{aligned}
\fthj \coloneqq \int_{I_n} (\Pia f - \p_t\Uth) \phi_j^n \,\dd t, & && \bxithj \coloneqq \int_{I_n} \nabla u_h \,\phi_j^n\,\dd t.
\end{aligned}
\end{equation}
It will be useful to keep in mind that $\gtautha = \sum_{j=0}^{q_n} [ \psia \fthj - \nabla \psia \cdot \bxithj] \phi_j^n$ and that $\tautha = -\sum_{j=0}^{q_n} \psia \bxithj \phi_j^n$.
Let $\Gamma_{\ta}\coloneqq \{x\in \p\oma,\;\psia(x)=0\}$; note that if $\ta\in\calVhint$, then $\Gamma_{\ta} = \p\oma$, whereas $\Gamma_{\ta}$ is a strict subset of $\p\oma$ if $\ta\in\calVhext$.
We will now use Lemma~\ref{lem:p_robust_lifting} to show that, for each $0\leq j \leq q_n$, there exists $\bvthj  \in \Hdivoma \cap \RTNa$ such that
\begin{subequations}\label{eq:bvthj_properties}
\begin{gather}
\nabla{\cdot} \bvthj  = \psia \fthj - \nabla \psia \cdot \bxithj\quad \text{in }\oma,\qquad
\bvthj \cdot \bm{n}  =0 \quad \text{on }\Gamma_{\ta}, \\
\norm{\bvthj + \psia \bxithj}_{\oma} \lesssim \norm{ \Raj}_{H^{-1}(\oma)},
\end{gather}
\end{subequations}
where $\Raj \in H^{-1}(\oma)$ is defined by $\pair{\Raj}{v} \coloneqq (\fthj, v)_{\oma} - (\bxithj,\nabla v)_{\oma}$ for all $v\in H^1_0(\oma)$.
To check the hypotheses of Lemma~\ref{lem:p_robust_lifting}, we start by observing that the choice of $p_{\ta}$ in \eqref{eq:patch_polynomial_degree} implies that $\fthj \in \Pal$ and that $\bxithj \in \RTNal$ for all $0\leq j \leq q_n$.
For any interior vertex $\ta\in \calVhint$, it is seen from~\eqref{eq:gthan_mean_value_zero} that $(\fthj ,\psia)_{\oma}=(\bxithj,\nabla \psia)_{\oma}$ for all $0\leq j \leq q_n$. Therefore, the hypotheses of Lemma~\ref{lem:p_robust_lifting} are satisfied, and there exists $\bvthj \in  \Hdivoma \cap \RTNa$ satisfying~\eqref{eq:bvthj_properties}.

Next, we claim that $\bvthj \in \Va$ for all $0\leq j \leq q_n$. Indeed, the definition of $\Gamma_{\ta}$ implies that $\Gamma_{\ta}=\p\oma$ for all $\ta\in\calVhint$ and that $\p\oma\setminus \DO \subset \Gamma_{\ta}$ for all $\ta\in\calVhext$. Therefore, we have $\bvthj \in \Va$ for all $0\leq j \leq q_n$. It then follows that the function $\bm{v}_{\th}^{\ta,n} \coloneqq \sum_{j=0}^{q_n} \bvthj \,\phi_j^n  \in \Vthan $ and that this function satisfies
\begin{subequations}\label{eq:bvth_properties}
\begin{gather}
\nabla{\cdot} \bm{v}_{\th}^{\ta,n} = \sum_{j=0}^{q_n} [ \psia \fthj - \nabla \psia \cdot \bxithj] \phi_j^n = \gtautha, \\
\int_{I_n} \norm{\bm{v}_{\th}^{\ta,n} - \tautha }_{\oma}^2 \,\dd t = \sum_{j=0}^{q_n} \norm{\bvthj + \psia \bxithj}_{\oma}^2 \lesssim \sum_{j=0}^{q_n} \norm{\Raj}_{H^{-1}(\oma)}^2, \label{eq:bvth_prop_norm}
\end{gather}
\end{subequations}
where the equality in~\eqref{eq:bvth_prop_norm} results from the orthonormality of $\{\phi_j^n\}_{j=0}^{q_n}$.
We now claim that
\begin{equation}\label{eq:stability_temporal_decomposition}
\left\{\sum_{j=0}^{q_n} \norm{\Raj}_{H^{-1}(\oma)}^2 \right\}^{\frac{1}{2}} \leq \sup_{v \in \calQ_{q_n}(I_n;H^1_0(\oma))\setminus\{0\}} \frac{\pair{\Ra}{v}}{\left(\int_{I_n}\norm{\nabla v}_{\oma}^2\,\dd t\right)^{\frac{1}{2}}}.
\end{equation}
For each $j=0,\dots,q_n$, we define $z_j \in H^1_0(\oma)$ by $(\nabla z_j,\nabla v)_{\oma} = \pair{\Raj}{v}$ for all $v\in H^1_0(\oma)$. It is then straightforward to show that $\norm{\nabla z_j}_\oma^2 = \pair{\Raj}{z_j} = \norm{\Raj}_{H^{-1}(\oma)}^2$ for each $j=0,\dots,q_n$.
Then, we define $z_*\in  \calQ_{q_n}(I_n;H^1_0(\oma))$ by $z_* \coloneqq \sum_{j=0}^{q_n} z_j \phi_{j}^n $. It follows from the orthonormality of the temporal basis that
$\int_{I_n}\norm{\nabla z_*}^2_\oma \,\dd t =  \sum_{j=0}^{q_n} \norm{\Raj}_{H^{-1}(\oma)}^2$.
Fubini's theorem and \eqref{eq:fthj_bxithj_def} imply that
\[
\begin{split}
\pair{\Ra}{z_*}
&= \sum_{j=0}^{q_n} \int_{I_n} (\Pia f - \p_t\Uth, \phi_j^n  z_j)_{\oma} -( \nabla u_h,\phi_j^n  \nabla z_j)_{\oma} \,\dd t \\
& = \sum_{j=0}^{q_n} \left\{ \left( {\textstyle\int}_{I_n} (\Pia f - \p_t\Uth \phi_j^n\,\dd t,  z_j\right)_{\oma} - \left({\textstyle \int}_{I_n} \nabla u_h \phi_j^n\,\dd t, \nabla  z_j\right)_{\oma} \right\}\\
& = \sum_{j=0}^{q_n} \left\{ ( \fthj , z_j)_{\oma} - (\bxithj, \nabla z_j)_{\oma} \right\}
= \sum_{j=0}^{q_n} \pair{\Raj}{z_j} = \sum_{j=0}^{q_n} \norm{\Raj}_{H^{-1}(\oma)}^2.
\end{split}
\]
Hence, the above identities immediately imply \eqref{eq:stability_temporal_decomposition}.
Therefore, we combine~\eqref{eq:bvth_properties} and \eqref{eq:stability_temporal_decomposition} to deduce that $\bm{v}_{\th}^{\ta,n} \in \Vthan$ satisfies $\nabla{\cdot} \bm{v}_{\th}^{\ta,n} = \gtautha$ and $\int_{I_n}\norm{\bm{v}_h - \tautha}_{\oma}^2\,\dd t$ is bounded by the right-hand side of \eqref{eq:flux_reconstruction_stability}. This implies \eqref{eq:flux_reconstruction_stability} as explained above.
\quad\end{proof}

\subsection{Local efficiency} We can now prove the local efficiency bound~\eqref{eq:Y_norm_guaranteed_efficiency_lower_local}.

\paragraph{Proof of the local efficiency bound~\eqref{eq:Y_norm_guaranteed_efficiency_lower_local}}
Consider a time-step $I_n$ and an element $K\in\calT^n$. First, note that $[\etaJ]^2 \leq \sum_{\ta\in \calV_K} \Eya^2$ trivially, where we recall that $\calV_K$ denotes the set of vertices of $K$. Hence, it remains only to bound $\int_{I_n} [\etaEq]^2 \,\dd t$. To this end, observe that $\sth|_{K\times I_n} = \sum_{\ta\in\calV_K} \stha|_{K\times I_n}$, and that
\begin{align}
\int_{I_n} [\etaEq]^2 \,\dd t &= \int_{I_n} \norm { {\textstyle\sum}_{\ta\in \calV_K} (\stha + \psia \nabla\Uth)}_{K}^2\,\dd t \nonumber\\
& \leq (\abs{\calV_K}+1)   \int_{I_n} \sum_{\ta\in\calV_K} \norm{\stha + \psia \nabla u_{\th}}_K^2 + \norm{\nabla(u_{\th}-\Uth)}_K^2\,\dd t \nonumber \\
& \leq (\abs{\calV_K}+1)   \int_{I_n} \sum_{\ta\in\calV_K} \norm{\stha + \psia \nabla u_{\th}}_{\oma}^2 \,\dd t + [\etaJ]^2 ,\label{eq:Y_norm_local_efficiency_1}
\end{align}
where $\abs{\calV_K}$ is the number of vertices of the element $K$, which equals $d+1$ for simplices and where we have used that $\|{\cdot}\|_K \leq \|{\cdot}\|_{\oma}$ and the definition of $\etaJ$ in the last line.

Keeping in mind Lemma~\ref{lem:flux_reconstruction_stability}, we therefore turn our attention to bounding the dual norm of the patchwise residuals $\Ra$ for each $\ta\in \calV_K$. Consider an arbitrary $v \in \calQ_{q_n}(I_n,H^1_0(\oma))$ such that $\int_{I_n} \norm{\nabla v}_{\oma}^2 \,\dd t =1$; then~\eqref{eq:Y_formulation} implies that
\[
\begin{split}
\pair{\Ra}{v} &= \int_{I_n} \big(\Pia f, v \big)_{\oma} - \big(\p_t (\Uth) ,v \big)_{\oma}-\big(\nabla u_{\th},\nabla v\big)_\oma \,\dd t
\\ & = \int_{I_n} \pair{\p_t(u-\Uth)}{v} + (\nabla(u-\Uth),\nabla v)_{\oma}\, \dd t
 \\ &\qquad+\int_{I_n} (\nabla(\Uth-u_{\th}),\nabla v )_{\oma} - ( f-\Pia f,v)_{\oma}\,\dd t
 \\  & \eqqcolon E_1 + E_2 + E_3 + E_4.
\end{split}
\]
The definition of the $\|{\cdot}\|_{H^{-1}(\oma)}$-norm and the Cauchy--Schwarz inequality then yield $\abs{E_1+E_2+E_3} \lesssim \Eya$.
Finally, we find that $\abs{E_4} \leq \etaOsca$ where $\etaOsca$ is defined in~\eqref{eq:patch_data_oscillation}.
Therefore, we find that
\begin{equation}
\sup_{v\in \calQ_{q_n}(I_n,H^1_0(\oma))\setminus\{0\}} \frac{ \pair{\Ra}{v} }{ \left( \int_{I_n} \norm{\nabla v}^2_{\oma} \,\dd t \right)^{\frac{1}{2}} } \lesssim \left( \Eya^2 + [\etaOsca]^2 \right)^{\frac{1}{2}}.
\end{equation}
Recalling \eqref{eq:flux_reconstruction_stability} of Lemma~\ref{lem:flux_reconstruction_stability}, we deduce that, for each $\ta\in \calV_K$,
\begin{equation}\label{eq:Y_norm_local_efficiency_2}
\int_{I_n} \norm{\stha + \psia \nabla u_{\th}}_{\oma}^2\,\dd t \lesssim \Eya^2 + [\etaOsca]^2 ,
\end{equation}
which in combination with \eqref{eq:Y_norm_local_efficiency_1}, yields the desired result~\eqref{eq:Y_norm_guaranteed_efficiency_lower_local}.
\qquad\endproof

\subsection{Global efficiency}
We finally prove the global efficiency bounds~\eqref{eq:Y_norm_guaranteed_efficiency_lower_global} and \eqref{eq:Y_norm_alternative_efficiency}.
\paragraph{Proof of \eqref{eq:Y_norm_guaranteed_efficiency_lower_global} and \eqref{eq:Y_norm_alternative_efficiency}}
Recalling the definition~\eqref{eq:EY_norm_def_localized} of the localized seminorms $\abs{\cdot}_{\nEya}$, we claim that
\begin{multline}\label{eq:localization_inequality}
\sum_{\ta\in\calVh} \Eya^2 \lesssim \int_{I_n}\norm{\p_t(u-\Uth)}_{H^{-1}(\Om)}^2+\norm{\nabla(u-\Uth)}^2\,\dd t \\ +\int_{I_n}\norm{\nabla(u_{\th}-\Uth)}^2\,\dd t .
\end{multline}
The proof is essentially a counting argument after local Riesz mappings are
introduced to evaluate the negative norms $ \norm{\p_t(u-\Uth)}_{H^{-1}(\oma)}^2$ for all $\ta\in\calVh$, see also~\cite{BlechtaMalekVohralik2016}.
Summing~\eqref{eq:Y_norm_guaranteed_efficiency_lower_local} over $K\in\calT^n$ and using~\eqref{eq:localization_inequality} then leads to
\begin{multline}
\sum_{K\in\calT^n} \left[ \int_{I_n} [\etaEq]^2 \,\dd t + [\etaJ]^2  \right] \lesssim \int_{I_n} \norm{\p_t(u-\Uth)}_{H^{-1}(\Om)}^2+\norm{\nabla(u-\Uth)}^2 \,\dd t  \\+ \int_{I_n}\norm{\nabla(u_{\th}-\Uth)}^2\,\dd t
+ \sum_{\ta\in\calVh}  [\etaOsca]^2 . \label{eq:global_eff_both}
\end{multline}
Summing the bound~\eqref{eq:global_eff_both} for all $1\leq n\leq N$ immediately yields~\eqref{eq:Y_norm_guaranteed_efficiency_lower_global}, whereas \eqref{eq:Y_norm_alternative_efficiency} results from~\eqref{eq:global_eff_both} after invoking~\eqref{eq:main_jump_bound} and observing that $\norm{\calR_Y(\Uth)|_{I_n}}_{X^\prime}^2$ is bounded from above by $\int_{I_n} \norm{\p_t(u-\Uth)}_{H^{-1}(\Om)}^2+\norm{\nabla(u-\Uth)}^2 \,\dd t$.
\qquad\endproof

\section{Conclusion and outlook}
We have studied a posteriori error estimates for $hp$-$\tau q$ discretizations of parabolic problems based on arbitrarily high-order conforming Galerkin spatial discretizations and discontinuous Galerkin temporal discretizations. The equilibrated flux reconstructions lead to guaranteed upper bounds for the norm $\norm{u-u_{\th}}_{\nEy}$. Furthermore, the estimators have the key property of being unconditionally locally space-time efficient with respect to the local errors $\Eya$, with constants that are fully robust with respect to both the spatial and temporal approximation orders. The estimators are flexible in the sense that they do not require restrictive transition conditions on the refinement and coarsening between time-steps. We also showed that the composite norm of the error $\Ey$ is globally equivalent to $\norm{u-\Uth}_Y$ up to the minimum of coarsening error and data oscillation, with polynomial degree-robust constants in the equivalence.
Finally, the analysis given here can be extended in various directions: in \cite{ErnSmearsVohralik2016b}, we show that the equilibrated flux reconstruction employed here can also be used for obtaining a posteriori estimates for the $X$-norm of the error, with guaranteed upper bounds, and local space-time efficiency under the natural parabolic condition that $h^2\lesssim \tau$. Furthermore, the adaptation of Lemma~\ref{lem:p_robust_lifting} to the case of residual-based estimators is currently under investigation.


\end{document}